\documentclass[9pt,oneside,english]{amsart}
\usepackage[T1]{fontenc}
\usepackage[latin9]{inputenc}
\usepackage[a4paper]{geometry}
\usepackage{amsthm}
\usepackage{amstext}
\usepackage{amssymb}
\usepackage{esint}

\makeatletter

\numberwithin{equation}{section}
\numberwithin{figure}{section}
  \theoremstyle{plain}
  \newtheorem*{thm*}{\protect\theoremname}
  \theoremstyle{plain}
  \newtheorem*{cor*}{\protect\corollaryname}
\theoremstyle{plain}
\newtheorem{thm}{\protect\theoremname}[section]
  \theoremstyle{definition}
  \newtheorem{defn}[thm]{\protect\definitionname}
  \theoremstyle{plain}
  \newtheorem{prop}[thm]{\protect\propositionname}
  \theoremstyle{plain}
  \newtheorem{lem}[thm]{\protect\lemmaname}
  \theoremstyle{remark}
  \newtheorem{rem}[thm]{\protect\remarkname}

\usepackage{xypic}
\theoremstyle{definition}
\newtheorem{parn}{}[subsection]

\makeatother

\usepackage{babel}
  \providecommand{\corollaryname}{Corollary}
  \providecommand{\definitionname}{Definition}
  \providecommand{\lemmaname}{Lemma}
  \providecommand{\propositionname}{Proposition}
  \providecommand{\remarkname}{Remark}
  \providecommand{\theoremname}{Theorem}
\providecommand{\theoremname}{Theorem}

\begin{document}
\author{Adrien Dubouloz, David R. Finston, and Imad Jaradat} 
\address{Adrien Dubouloz\\ CNRS\\ Institut de Math\'{e}matiques de Bourgogne\\ Universit\'{e} de Bourgogne\\ 9 Avenue Alain Savary\\ BP 47870\\ 21078 Dijon Cedex\\ France} \email{Adrien.Dubouloz@u-bourgogne.fr} 
\address{David R. Finston\\ Mathematics Department\\ Brooklyn College, CUNY\\ 2900 Bedford Avenue\\ Brooklyn, NY 11210} \email{dfinston@brooklyn.cuny.edu} 
\address{ Imad Jaradat\\ Department of Mathematical Sciences\\ New Mexico State University\\ Las Cruces, New Mexico 88003} \email{imad\_jar@nmsu.edu} \thanks{Research supported in part by NSF Grant OISE-0936691 and ANR Grant "BirPol" ANR-11-JS01-004-01.}

\title{Proper triangular $\mathbb{G}_{a}$-actions on $\mathbb{A}^{4}$
are translations}
\begin{abstract}
We describe the structure of geometric quotients for proper locally
triangulable $\mathbb{G}_{a}$-actions on locally trivial $\mathbb{A}^{3}$-bundles
over a n\oe therian normal base scheme $X$ defined over a field
of characteristic $0$. In the case where $\dim X=1$, we show in
particular that every such action is a translation with geometric
quotient isomorphic to the total space of a vector bundle of rank
$2$ over $X$. As a consequence, every proper triangulable $\mathbb{G}_{a}$-action
on the affine four space $\mathbb{A}_{k}^{4}$ over a field of characteristic
$0$ is a translation with geometric quotient isomorphic to $\mathbb{A}_{k}^{3}$. 
\end{abstract}
\maketitle

\section*{Introduction}

The study of algebraic actions of the additive group $\mathbb{G}_{a}=\mathbb{G}_{a,\mathbb{C}}$
on complex affine spaces $\mathbb{A}^{n}=\mathbb{A}_{\mathbb{C}}^{n}$
has a long history which began in 1968 with a pioneering result of
Rentschler \cite{Ren68} who established that every such action on
the plane $\mathbb{A}^{2}$ is triangular in a suitable polynomial
coordinate system. Consequently, every fixed point free $\mathbb{G}_{a}$-action
on $\mathbb{A}^{2}$ is a translation, in the sense that the geometric
quotient $\mathbb{A}^{2}/\mathbb{G}_{a}$ is isomorphic to $\mathbb{A}^{1}$
and that $\mathbb{A}^{2}$ is equivariantly isomorphic to $\mathbb{A}^{2}/\mathbb{G}_{a}\times\mathbb{G}_{a}$
where $\mathbb{G}_{a}$ acts by translations on the second factor. 

Arbitrary $\mathbb{G}_{a}$-actions turn out to be no longer triangulable
in higher dimensions \cite{Bass84}. But the question whether a fixed
point free $\mathbb{G}_{a}$-action on $\mathbb{A}^{3}$ is a translation
or not was settled affirmatively, first for triangulable actions by
Snow \cite{Snow88} in 1988, then by Deveney and the second author
\cite{DevFin94a} in 1994 under the additional assumption that the
action is proper and then in general by Kaliman \cite{Kal04} in 2004.
The argument for triangulable actions depends on their explicit form
in an appropriate coordinate system which is used to check that the
algebraic quotient $\pi:\mathbb{A}^{3}\rightarrow\mathbb{A}^{3}/\!/\mathbb{G}_{a}={\rm Spec}(\Gamma(\mathbb{A}^{3},\mathcal{O}_{\mathbb{A}^{3}})^{\mathbb{G}_{a}})$
is a geometric quotient and that $\mathbb{A}^{3}/\!/\mathbb{G}_{a}$
is isomorphic to $\mathbb{A}^{2}$. For proper actions, the properness
implies that the geometric quotient $\mathbb{A}^{3}/\mathbb{G}_{a}$,
which a priori only exists as an algebraic space, is separated whence
a scheme by virtue of Chow's Lemma. This means equivalently that the
$\mathbb{G}_{a}$-action is not only locally equivariantly trivial
in the \'etale topology but in fact locally trivial in the Zariski
topology, i.e. that $\mathbb{A}^{3}$ is covered by invariant Zariski
affine open subsets of the from $V_{i}=U_{i}\times\mathbb{G}_{a}$
on which $\mathbb{G}_{a}$ acts by translations on the second factor.
Since $\mathbb{A}^{3}$ is factorial, the open subsets $V_{i}$ can
even be chosen to be principal, which implies in turn that $\mathbb{A}^{3}/\mathbb{G}_{a}$
is a quasi-affine scheme, in fact an open subset of $\mathbb{A}^{3}/\!/\mathbb{G}_{a}\simeq\mathbb{A}^{2}$
with at most finite complement. The equality $\mathbb{A}^{3}/\mathbb{G}_{a}=\mathbb{A}^{3}/\!/\mathbb{G}_{a}$
ultimately follows by comparing Euler characteristics. Kaliman's general
proof proceeds along a completely different approach, drawing on topological
arguments to show directly that the algebraic quotient morphism $\pi:\mathbb{A}^{3}\rightarrow\mathbb{A}^{3}/\!/\mathbb{G}_{a}$
is a locally trivial $\mathbb{A}^{1}$-bundle. 

Kaliman's result can be reinterpreted as the striking fact that the
topological contractiblity of $\mathbb{A}^{3}$ is a strong enough
constraint to guarantee that a fixed point free $\mathbb{G}_{a}$-action
on it is automatically proper. This implication fails completely in
higher dimensions where non proper fixed point free $\mathbb{G}_{a}$-actions
abound, even in the case of triangular actions on $\mathbb{A}^{4}$
as illustrated by Deveney-Finston-Gehrke in \cite{DevFinGe94}. Starting
from dimension $5$, it is known that properness and triangulability
are no longer enough to imply global equivariant triviality or at
least local equivariant triviality in the Zariski topology, as shown
by examples of triangular actions on $\mathbb{A}^{5}$ with either
strictly quasi-affine geometric quotients or with geometric quotients
existing only as separated algebraic spaces constructed respectively
by Winkelmann \cite{Win90} and Deveney-Finston \cite{DevFin95}.
\\

But the question whether a proper $\mathbb{G}_{a}$-action on $\mathbb{A}^{4}$
is a translation or is at least  locally equivariantly trivial in
the Zariski topology remains open. Very little progress had been made
in the study of these actions during the last decades, and the only
existing partial results so far concern triangular actions: Deveney,
van Rossum and the second author \cite{DevFinvR04} established in
2004 that a Zariski locally equivariantly trivial triangular $\mathbb{G}_{a}$-action
on $\mathbb{A}^{4}$ is a translation. The proof depends on the finite
generation of the ring of invariants for such actions established
by Daigle-Freudenburg \cite{DaiFreu01} and exploits the very particular
structure of these rings. Incidentally, it is known in general that
local triviality for a proper action on $\mathbb{A}^{n}$ follows
from the finite generation and regularity of the ring of invariants.
But even knowing the former for triangular actions on $\mathbb{A}^{4}$,
a direct proof of the latter condition remains elusive. The second
positive result concerns a special type of triangular $\mathbb{G}_{a}$-actions
generated by derivations of $\mathbb{C}[x,y,z,u]$ of the form $r(x)\partial_{y}+q(x,y)\partial_{z}+p(x,y)\partial_{u}$
where $r(x)\in\mathbb{C}[x]$ and $p(x,y),q(x,y)\in\mathbb{C}[x,y,]$.
To insist on the fact that $p(x,y)$ belongs to $\mathbb{C}[x,y]$
and not only to $\mathbb{C}[x,y,z]$ as it would be the case for a
general triangular situation, these derivations (and the $\mathbb{G}_{a}$-actions
they generate) were named \emph{twin-triangular} in \cite{DevFin02}.
The case where $r(x)$ has simple roots was first settled in 2002
by Deveney and the second author in \emph{loc. cit.} by explicitly
computing the invariant ring $\mathbb{C}[x,y,z,u]^{\mathbb{G}_{a}}$
and investigating the structure of the algebraic quotient morphism
$\mathbb{A}^{4}\rightarrow\mathbb{A}^{4}/\!/\mathbb{G}_{a}=\mathrm{Spec}(\mathbb{C}[x,y,z_{1},z_{2}]^{\mathbb{G}_{a}})$.
The simplicity of the roots of $r(x)$ was crucial to achieve the
computation, and the generalization of the result to arbitrary twin-triangular
actions obtained in 2012 by the first two authors \cite{DubFin11}
required completely different methods which focused more on the nature
of the corresponding geometric quotients $\mathbb{A}_{\mathbb{C}}^{4}/\mathbb{G}_{a}$.
The latter a priori exist only as separated algebraic spaces and the
crucial step in \emph{loc. cit.} was to show that for twin-triangular
actions they are in fact schemes, or, equivalently that proper twin-triangular
$\mathbb{G}_{a}$-actions on $\mathbb{A}^{4}$ are not only locally
equivariantly trivial in the \'etale topology but also in the Zariski
topology. This enabled in turn the use of the aforementioned result
of Deveney-Finston-van Rossum to conclude that such actions are indeed
translations. 

One of the main obstacles to extend the above results to arbitrary
triangular actions comes from the fact that in contrast with fixed
point freeness, the property for a triangular $\mathbb{G}_{a}$-action
on $\mathbb{A}^{4}$ to be proper is in general subtle to characterize
effectively in terms of its associated locally nilpotent derivation.
A good illustration of these difficulties is given by the following
family of fixed point free $\mathbb{G}_{a}$-actions 
\[
\sigma_{r}:\mathbb{G}_{a}\times\mathbb{A}^{4}\rightarrow\mathbb{A}^{4},\;\left(t,(x,y,z,u)\right)\mapsto(x,y+tx^{2},z+2yt+x^{2}t^{2},u+(1+x^{r}z)t+x^{r}yt^{2}+\frac{1}{3}x^{r+1}t^{3})\quad r\geq1,
\]
generated by the triangular derivations $\delta_{r}=x^{2}\partial_{y}+2y\partial_{z}+(1+x^{r}z)\partial_{u}$
of $\mathbb{C}[x,y,z,u]$, which are either non proper if $r=1,2$
or translations otherwise. The fact that $\sigma_{r}$ is a translation
for every $r\geq4$ follows immediately from the observation that
$\delta_{r}$ admits the variable $s=u-x^{r-2}yz+\frac{2}{3}x^{r-4}y^{3}$
as a global slice. The case $r=3$ is slightly more complicated: one
can first observe that $\delta_{3}$ is conjugated via the triangular
change of variable $\tilde{u}=u-x^{r-2}yz$ to the twin-triangular
derivation $x^{2}\partial_{y}+2y\partial_{z}+(1-2xy^{2})\partial_{\tilde{u}}$
of $\mathbb{C}[x,y,z,\tilde{u}]$. The projection $\mathrm{pr}_{x,y,\tilde{u}}:\mathbb{A}^{4}\rightarrow\mathbb{A}^{3}$
is then equivariant for the fixed point free $\mathbb{G}_{a}$-action
on $\mathbb{A}^{3}$ generated by the triangular derivation $x^{2}\partial_{y}+(1-2xy^{2})\partial_{\tilde{u}}$
of $\mathbb{C}[x,y,\tilde{u}]$ and it descends to a locally trivial
$\mathbb{A}^{1}$-bundle $\rho:\mathbb{A}^{4}/\mathbb{G}_{a}\rightarrow\mathbb{A}^{3}/\mathbb{G}_{a}\simeq\mathbb{A}^{2}$
between the respective geometric quotients. Since $\mathbb{A}^{2}$
is affine and factorial, $\rho$ is a trivial $\mathbb{A}^{1}$-bundle
and hence the $\mathbb{G}_{a}$-action generated by $\delta_{3}$
is a translation. On the other hand, the non properness of $\sigma_{2}$
can be seen quickly via the invariant hypersurface method outlined
in \cite{DubFin11}, namely, one checks in this case by a direct computation
that the induced $\mathbb{G}_{a}$-action on the invariant hypersurface
$H=\left\{ x^{2}z=y^{2}-\frac{3}{2}\right\} \subset\mathbb{A}^{4}$
is not proper, with non separated geometric quotient $H/\mathbb{G}_{a}$
isomorphic to the product of the affine line $\mathbb{A}^{1}$ with
the affine line with a double origin. The failure of properness in
the case where $r=1$ is even more subtle to analyze since in contrast
with the previous case, the induced action on every invariant hypersurface
of the form $H_{\lambda}=\left\{ x^{2}z=y^{2}-\lambda\right\} $,
$\lambda\in\mathbb{C}$, turns out to be proper. Going back to the
definition of the properness for the action $\sigma_{1}$, which says
that the morphism $\Phi=(\mathrm{pr}_{2},\sigma_{1}):\mathbb{G}_{a}\times\mathbb{A}^{4}\rightarrow\mathbb{A}^{4}\times\mathbb{A}^{4}$
is proper, one can argue that the union of the following sequence
of points 
\[
(p_{n},q_{n})=(p_{n};\mu_{1}(\sqrt{n^{3}},p_{n}))=((\frac{\sqrt[3]{6}}{n},-\frac{\sqrt[3]{36}}{2\sqrt{n}},\frac{1}{\sqrt[3]{6}\sqrt{n}},0);(\frac{\sqrt[3]{6}}{n},\frac{\sqrt[3]{36}}{2\sqrt{n}},\frac{1}{\sqrt[3]{6}\sqrt{n}},1))\in\mathbb{A}^{4}\times\mathbb{A}^{4},\quad n\in\mathbb{N}
\]
and its limit $(p_{\infty};q_{\infty})=(p_{\infty},\mu_{1}(1,p_{\infty}))=\left((0,0,0,0);(0,0,0,1)\right)$
is a compact subset of $\mathbb{A}^{4}\times\mathbb{A}^{4}$ equipped
with the analytic topology whose inverse image by $\Phi$ is unbounded.
So $\Phi$ is not proper as an analytic map between the corresponding
varieties equipped with their respective underlying structures of
analytic manifolds and hence is not proper in the algebraic category
either.\\

In this article, we reconsider proper triangular actions on $\mathbb{A}^{4}$
in broader framework and we develop new techniques to overcome the
above difficulties. These enable in turn to completely solve the question
of global equivariant triviality for such actions. Since a triangular
$\mathbb{G}_{a}$-action on $\mathbb{A}^{4}=\mathrm{Spec}(\mathbb{C}[x,y,z,u])$
preserves the variable $x$, it can be considered as an action of
the additive group scheme $\mathbb{G}_{a,\mathbb{C}[x]}=\mathbb{G}_{a}\times_{\mathrm{Spec}(\mathbb{C})}\mathrm{Spec}(\mathbb{C}[x])$
on the affine $3$-space $\mathbb{A}_{\mathbb{C}[x]}^{3}$ over $\mathrm{Spec}(\mathbb{C}[x])$
so that the setup is in fact $3$-dimensional over a parameter space.
The properties for a $\mathbb{G}_{a,\mathbb{C}[x]}$-action on $\mathbb{A}_{\mathbb{C}[x]}^{3}$
to be proper or triangulable being both local on the parameter space,
a cost free generalization is obtained by replacing $\mathrm{Spec}(\mathbb{C}[x])$
by an arbitrary n\oe therian normal scheme $X$ defined over a field
of characteristic zero and the trivial $\mathbb{A}^{3}$-bundle $\mathrm{pr}_{x}:\mathbb{A}_{\mathbb{C}[x]}^{3}\rightarrow\mathrm{Spec}(\mathbb{C}[x])$
of $\mathrm{Spec}(\mathbb{C}[x])$ by a Zariski locally trivial $\mathbb{A}^{3}$-bundle
$\pi:E\rightarrow X$. Our main result then reads as follows: 
\begin{thm*}
Let $X$ be a n\oe therian normal scheme defined over a field of
characteristic zero, let $\pi:E\rightarrow X$ be a Zariski locally
trivial $\mathbb{A}^{3}$-bundle equipped with a proper locally triangulable
$\mathbb{G}_{a,X}$-action and let $\mathrm{p}:\mathfrak{X}=E/\mathbb{G}_{a,X}\rightarrow X$
be the geometric quotient taken in the category of algebraic $X$-spaces.
Then there exists an open sub-scheme $U$ of $X$ with $\mathrm{codim}_{X}(X\setminus U)\geq2$
such that $\mathfrak{X}_{U}=\mathrm{p}^{-1}(U)\rightarrow U$ has
the structure of a Zariski locally trivial $\mathbb{A}^{2}$-bundle. 
\end{thm*}
Note in particular that since in the original problem, the base $X=\mathrm{Spec}(\mathbb{C}[x])$
is $1$-dimensional, this Theorem and an appeal to the aforementioned
result \cite{DevFinvR04} are enough to settle the question for $\mathbb{A}_{\mathbb{C}}^{4}$.
The conclusion of the above Theorem is essentially optimal. Indeed,
in the example due to Winkelmann \cite{Win90}, one has $X=\mathrm{Spec}(\mathbb{C}[x,y])$,
$\pi=\mathrm{pr}_{x,y}:\mathbb{A}_{X}^{3}=\mathrm{Spec}(\mathbb{C}[x,y][u,v,w])\rightarrow X$
equipped with the proper triangular $\mathbb{G}_{a,X}$-action generated
by the $\mathbb{C}[x,y]$-derivation $\partial=x\partial_{u}+y\partial_{v}+(1+xv-yu)\partial_{w}$
of $\mathbb{C}[x,y][u,v,w]$, and the geometric quotient $\mathrm{p}:\mathfrak{X}=\mathbb{A}_{X}^{3}/\mathbb{G}_{a,X}\rightarrow X$
is the strictly quasi-affine complement of the closed subset $\left\{ x=y=z=0\right\} $
in the $4$-dimensional smooth affine quadric $Q\subset\mathbb{A}_{X}^{3}$
with equation $xt_{2}+yt_{1}=z(z+1)$. The structure morphism $\mathrm{p}:\mathfrak{X}\rightarrow X$
is easily seen to be an $\mathbb{A}^{2}$-fibration, which restricts
to a locally trivial $\mathbb{A}^{2}$-bundle over the open subset
$U=X\setminus\{(0,0)\}$. However, there is no Zariski or \'etale
open neighborhood of the origin $(0,0)\in X$ over which $\mathrm{p}:\mathfrak{X}\rightarrow X$
restricts to a trivial $\mathbb{A}^{2}$-bundle for otherwise $\mathrm{p}:\mathfrak{X}\rightarrow X$
would be an affine morphism and so $\mathfrak{X}$ would be an affine
scheme. The situation for the $\mathbb{C}[x,y]$-derivation $\partial=x\partial_{u}+y\partial_{v}+(1+xv^{2})\partial_{w}$
of $\mathbb{C}[x,y][u,v,w]$ constructed by Deveney-Finston \cite{DevFin95}
is very similar: here the geometric quotient $\mathfrak{X}=\mathbb{A}_{X}^{3}/\mathbb{G}_{a,X}$
is a separated algebraic space which is not a scheme and the structure
morphism $\mathrm{p}:\mathfrak{X}\rightarrow X$ is again an $\mathbb{A}^{2}$-fibration
restricting to a Zariski locally trivial $\mathbb{A}^{2}$-bundle
over $U=X\setminus\{(0,0)\}$ but whose restriction to any Zariski
or \'etale open neighborhood of the origin $(0,0)\in X$ is nontrivial.
\\

\noindent In contrast, in the case of a $1$-dimensional affine base,
we can immediately derive the following Corollaries: 
\begin{cor*}
Let $\pi:E\rightarrow S$ be a rank $3$ vector bundle over an affine
Dedekind scheme $S=\mathrm{Spec}(A)$ defined over a field $k$ of
characteristic $0$. Then every proper locally triangulable $\mathbb{G}_{a,S}$-action
on $E$ is equivariantly trivial with geometric quotient $E/\mathbb{G}_{a,S}$
isomorphic to a vector bundle of rank $2$ over $S$, stably isomorphic
to $E$. \end{cor*}
\begin{proof}
By the previous Theorem, the geometric quotient $\mathrm{p}:E/\mathbb{G}_{a,S}\rightarrow S$
has the structure of a Zariski locally trivial $\mathbb{A}^{2}$-bundle,
hence is a vector bundle of rank $2$ by \cite{BCW77}. In particular,
$E/\mathbb{G}_{a,S}$ is affine which implies in turn that $\rho:E\rightarrow E/\mathbb{G}_{a,S}$
is a trivial $\mathbb{G}_{a,S}$-bundle. So $E\simeq E/\mathbb{G}_{a,S}\times_{S}\mathbb{A}_{S}^{1}$
as vector bundles over $S$. \end{proof}
\begin{cor*}
Let $S=\mathrm{Spec}(A)$ be an affine Dedekind scheme defined over
a field of characteristic $0$. Then every proper triangular $\mathbb{G}_{a,S}$-action
on $\mathbb{A}_{S}^{3}$ is a translation. \end{cor*}
\begin{proof}
By the previous Corollary, $\mathbb{A}_{S}^{3}/\mathbb{G}_{a,S}$
is a stably trivial vector bundle of rank $2$ over $S$, whence is
isomorphic to the trivial bundle $\mathbb{A}_{S}^{2}$ over $S$ by
virtue of \cite[IV 3.5]{Bas68}. 
\end{proof}
\noindent Coming back to the original problem for triangular $\mathbb{G}_{a,k}$-actions
on $\mathbb{A}_{k}^{4}$, the previous Corollary does in fact eliminate
the need for \cite{DevFinvR04} hence the dependency on the fact that
the corresponding rings of invariants are finitely generated: 
\begin{cor*}
If $k$ is a field of characteristic $0$, then every proper triangular
$\mathbb{G}_{a,k}$-action on $\mathbb{A}_{k}^{4}$ is a translation.\end{cor*}
\begin{proof}
Letting $\mathbb{A}_{k}^{4}=\mathrm{Spec}(k[x,y,z,u])$, we may assume
that the action is generated by a $k$-derivation of the form $\partial=r(x)\partial_{y}+q(x,y)\partial_{z}+p(x,y,z)\partial_{u}$.
As explained above, the latter can be considered as a triangular $k[x]$-derivation
of $k[x][y,z,u]$ generating a proper $\mathbb{G}_{a,k[x]}$-action
on $\mathbb{A}_{k}^{4}=\mathbb{A}_{k[x]}^{3}$ which is, by the previous
Corollary, a trivial $\mathbb{G}_{a}$-bundle over its geometric quotient
$\mathbb{A}_{k}^{4}/\mathbb{G}_{a,k}\simeq\mathbb{A}_{k[x]}^{3}/\mathbb{G}_{a,k[x]}\simeq\mathbb{A}_{k[x]}^{2}\simeq\mathbb{A}_{k}^{3}$. 
\end{proof}
Let us now briefly explain the general philosophy behind the proof.
After localizing at codimension $1$ points of $X$, the Main Theorem
reduces to the statement that a proper $\mathbb{G}_{a,S}$-action
$\sigma:\mathbb{G}_{a,S}\times_{S}\mathbb{A}_{S}^{3}\rightarrow\mathbb{A}_{S}^{3}$
on the affine affine space $\mathbb{A}_{S}^{3}=\mathrm{Spec}(A[y,z,u])$
over the spectrum of a discrete valuation ring, generated by a triangular
$A$-derivation $\partial=a\partial_{y}+q(y)\partial_{z}+p(y,z)\partial_{u}$
of $A[y,z,u]$, where $a\in A\setminus\left\{ 0\right\} $, $q(y)\in A[y]$
and $p(y,z)\in A[y,z]$, is a translation. Triangularity immediately
implies that the restriction of $\sigma$ to the generic fiber of
$\mathrm{pr}_{S}:\mathbb{A}_{S}^{3}\rightarrow S$ is a translation
with $a^{-1}y$ as a global slice. This reduces the problem to the
study of neighborhoods of points of the geometric quotient $\mathfrak{X}=\mathbb{A}_{S}^{3}/\mathbb{G}_{a,S}$
supported on the closed fiber of the structure morphism $\mathrm{p}:\mathfrak{X}\rightarrow S$.
A second feature of triangularity is that $\sigma$ commutes with
the action $\tau:\mathbb{G}_{a,S}\times_{S}\mathbb{A}_{S}^{3}\rightarrow\mathbb{A}_{S}^{3}$
generated by the $A$-derivation $\partial_{u}$ which therefore descends
to a $\mathbb{G}_{a,S}$-action $\overline{\tau}$ on the geometric
quotient $\mathfrak{X}=\mathbb{A}_{S}^{3}/\mathbb{G}_{a,S}$. On the
other hand, $\sigma$ descends via the projection $\mathrm{pr}_{y,z}:\mathbb{A}_{S}^{3}\rightarrow\mathbb{A}_{S}^{2}=\mathrm{Spec}(A[y,z])$
to the action $\overline{\sigma}$ on $\mathbb{A}_{S}^{2}$ generated
by the $A$-derivation $\overline{\partial}=a\partial_{y}+q(y)\partial_{z}$
of $A[y,z]$. Even though $\overline{\sigma}$ and $\overline{\tau}$
are no longer fixed point free in general, if we take the quotient
of $\mathbb{A}_{S}^{2}$ by the action $\overline{\sigma}$ as an
algebraic stack $[\mathbb{A}_{S}^{2}/\mathbb{G}_{a,S}]$ we obtain
a cartesian square \[\xymatrix{ \mathbb{A}^3_S \ar[d]_{\mathrm{pr}_{y,z}} \ar[r] & \mathfrak{X}=\mathbb{A}^3_S/\mathbb{G}_{a,S} \ar[d] \\ \mathbb{A}^2_S \ar[r] & [\mathbb{A}^2_S/\mathbb{G}_{a,S}]}\]which
simultaneously identifies the quotient stacks $[\mathbb{A}_{S}^{2}/\mathbb{G}_{a,S}]$
for the action $\overline{\sigma}$ and $[\mathfrak{X}/\mathbb{G}_{a,S}]$
for the action $\overline{\tau}$ with the quotient stack of $\mathbb{A}_{S}^{3}$
for the $\mathbb{G}_{a,S}^{2}$-action defined by the commuting actions
$\sigma$ and $\tau$. In this setting, the global equivariant triviality
of the action $\sigma$ becomes equivalent to the statement that a
separated algebraic $S$-space $\mathfrak{X}$ admitting a $\mathbb{G}_{a,S}$-action
whose algebraic stack quotient $[\mathfrak{X}/\mathbb{G}_{a,S}]$
is isomorphic to that of a triangular $\mathbb{G}_{a,S}$-action on
$\mathbb{A}_{S}^{2}$ is an affine scheme. 

While a direct proof of this reformulation seems totally out of reach
with existing methods, it turns out that its conclusion holds over
a certain $\mathbb{G}_{a,S}$-invariant principal open subset $V$
of $\mathbb{A}_{S}^{2}$ which dominates $S$ and for which the algebraic
stack quotient $[V/\mathbb{G}_{a,S}]$ is in fact represented by a
locally separated algebraic sub-space of $[\mathbb{A}_{S}^{2}/\mathbb{G}_{a,S}]$.
This provides at least an affine open subscheme $V\times_{S}\mathbb{A}_{S}^{1}/\mathbb{G}_{a,S}$
of $\mathfrak{X}$ dominating $S$, and leaves us with a closed subset
of codimension at most $2$ of $\mathfrak{X}$, supported on the closed
fiber of $\mathrm{p}:\mathfrak{X}\rightarrow S$, in a neighborhood
of which no further information is a priori available to decide even
the schemeness of $\mathfrak{X}$. But similar to the argument in
\cite{DubFin11}, this situation can be rescued for twin-triangular
actions: the fact that for such actions $\partial u=p(y,z)$ is actually
a polynomial in $y$ only enables the same reasoning with respect
to the other projection $\mathrm{pr}_{y,u}:\mathbb{A}_{S}^{3}\rightarrow\mathbb{A}_{S}^{2}=\mathrm{Spec}(A[y,u])$,
yielding a second affine open sub-scheme $V'\times_{S}\mathbb{A}_{S}^{1}/\mathbb{G}_{a,S}$
of $\mathfrak{X}$ dominating $S$. This implies at least the schemeness
of $\mathfrak{X}$, provided that the open subsets $V$ and $V'$
can be chosen so that the union of the corresponding open subschemes
of $\mathfrak{X}$ covers the closed fiber of $\mathrm{p}:\mathfrak{X}\rightarrow S$.
\\

The scheme of the article is the following. The first two sections
recall basic notions and discuss a couple of preliminary technical
reductions. The third section is devoted to establishing an effective
criterion for non properness of fixed point free triangular actions
from which we deduce the intermediate fact that every proper triangular
action is twin-triangulable. Then in the next section, we establish
that proper twin-triangular actions are indeed translations. Here,
in contrast with the proof for the complex case given in \cite{DubFin11},
our argument is independent of finite generation of rings of invariants
and reduces the systematic study of algebraic spaces quotients to
a minimum thanks to an appropriate Sheshadri cover trick \cite{Sesh72}.

\section{Recollection on proper, fixed point free and locally triangulable
$\mathbb{G}_{a}$-actions }

\subsection{Proper versus fixed point free actions}

\indent\newline\noindent Recall that an action $\sigma:\mathbb{G}_{a,S}\times_{S}E\rightarrow E$
of the additive group scheme $\mathbb{G}_{a,S}=\mathrm{Spec}_{S}(\mathcal{O}_{S}[t])=S\times_{\mathbb{Z}}\mathrm{Spec}(\mathbb{Z}[t])$
on an $S$-scheme $E$ is called proper if the morphism $\Phi=(\mathrm{pr}_{2},\sigma):\mathbb{G}_{a,S}\times_{S}E\rightarrow E\times_{S}E$
is proper. 

\begin{parn} If $S$ is moreover defined over a field $k$ of characteristic
zero, then the fact that $\mathbb{G}_{a,k}$ is affine and has no
nontrivial algebraic subgroups implies that properness is equivalent
to $\Phi$ being a closed immersion. In particular, a proper $\mathbb{G}_{a,S}$-action
is in this case fixed point free and as such, is equivariantly locally
trivial in the \'etale topology on $E$. That is, there exists an
affine $S$-scheme $U$ and a surjective \'etale morphism $f:V=U\times_{S}\mathbb{G}_{a,S}\rightarrow E$
which is equivariant for the action of $\mathbb{G}_{a,S}$ on $U\times_{S}\mathbb{G}_{a,S}$
by translations on the second factor. This implies in turn the existence
of a geometric quotient $\rho:E\rightarrow\mathfrak{X}=E/\mathbb{G}_{a,S}$
in the form of an \'etale locally trivial principal $\mathbb{G}_{a,S}$-bundle
over an algebraic $S$-space $\mathrm{p}:\mathfrak{X}\rightarrow S$
(see e.g. \cite[10.4]{LMB00}). Informally, $\mathfrak{X}$ is the
quotient of $U$ by the \'etale equivalence relation which identifies
two points $u,u'\in U$ whenever there exists $t,t'\in\mathbb{G}_{a,S}$
such that $f(u,t)=f(u',t')$.

\end{parn}

\begin{parn} \label{par:Properness_charac} Conversely, a fixed point
free $\mathbb{G}_{a,S}$-action is proper if and only if the geometric
quotient $\mathfrak{X}=E/\mathbb{G}_{a,S}$ is a separated $S$-space.
Indeed, by definition $\mathrm{p}:\mathfrak{X}\rightarrow S$ is separated
if and only if the diagonal morphism $\Delta:\mathfrak{X}\rightarrow\mathfrak{X}\times_{S}\mathfrak{X}$
is a closed immersion, a property which is local on the target with
respect to the fpqc topology \cite[II.3.8]{Knu71} and \cite[VIII.5.5]{SGA1}.
Since $\rho:E\rightarrow\mathfrak{X}$ is a $\mathbb{G}_{a,S}$-bundle,
taking the fpqc base change by $\rho\times\rho:E\times_{S}E\rightarrow\mathfrak{X}\times_{S}\mathfrak{X}$
yields a cartesian square \[\xymatrix{\mathbb{G}_{a,S} \times_S E \ar[r]^{\Phi} \ar[d]_{\rho \circ \mathrm{pr}_2} & E \times_S E \ar[d]^{\rho\times\rho} \\ \mathfrak{X} \ar[r]^-{\Delta} & \mathfrak{X} \times_S \mathfrak{X} }\]from
which we see that $\Delta$ is a closed immersion if and only if $\Phi$
is. 

\end{parn}

\subsection{Locally triangulable actions}

\indent\newline\noindent Given an affine scheme $S=\mathrm{Spec}(A)$
defined over a field of characteristic zero, an action $\sigma:\mathbb{G}_{a,S}\times_{S}\mathbb{A}_{S}^{n}\rightarrow\mathbb{A}_{S}^{n}$
generated by a locally nilpotent $A$-derivation $\partial$ of $\Gamma(\mathbb{A}_{S}^{n},\mathcal{O}_{\mathbb{A}_{S}^{n}})$
is called \emph{triangulable} if there exists an isomorphism of $A$-algebras
$\tau:\Gamma(\mathbb{A}_{A}^{n},\mathcal{O}_{\mathbb{A}_{A}^{n}})\stackrel{\sim}{\rightarrow}A[x_{1},\cdots,x_{n}]$
such that the conjugate $\delta=\tau\circ\partial\circ\tau^{-1}$
of $\partial$ is triangular with respect to the ordered coordinate
system $(x_{1},\ldots,x_{n})$, i.e. has the form 
\[
\delta=p_{0}\frac{\partial}{\partial x_{1}}+\sum_{i=1}^{n}p_{i-1}(x_{1},\ldots,x_{i-1})\frac{\partial}{\partial x_{i}}
\]
where $p_{0}\in A$ and where for every $i=1,\ldots,n$, $p_{i-1}(x_{1},\ldots,x_{i-1})\in A[x_{1},\ldots,x_{i-1}]\subset A[x_{1},\ldots,x_{n}]$.
By localizing this notion over the base $S$, we arrive at the following
definition:
\begin{defn}
Let $X$ be a scheme defined over a field of characteristic zero and
let $\pi:E\rightarrow X$ be a Zariski locally trivial $\mathbb{A}^{n}$-bundle
over $X$. An action $\sigma:\mathbb{G}_{a,X}\times_{X}E\rightarrow E$
of $\mathbb{G}_{a,X}$ on $E$ is called \emph{locally triangulable}
if there exists a covering of $\mathrm{Spec}(A)$ by affine open sub-schemes
$S_{i}=\mathrm{Spec}(A_{i})$, $i\in I$, such that $E\mid_{S_{i}}\simeq\mathbb{A}_{S_{i}}^{n}$
and such that the $\mathbb{G}_{a,S_{i}}$-action $\sigma_{i}:\mathbb{G}_{a,S_{i}}\times_{S_{i}}\mathbb{A}_{S_{i}}^{n}\rightarrow\mathbb{A}_{S_{i}}^{n}$
on $\mathbb{A}_{S_{i}}^{n}$ induced by $\sigma$ is triangulable. 
\end{defn}
A Zariski locally trivial $\mathbb{A}^{1}$-bundle $\pi:E\rightarrow X$
equipped with a fixed point free $\mathbb{G}_{a,X}$-action is nothing
but a principal $\mathbb{G}_{a,X}$-bundle. As mentioned in the introduction,
the nature of fixed point free locally triangulable $\mathbb{G}_{a,X}$-actions
on Zariski locally trivial $\mathbb{A}^{2}$-bundles $\pi:E\rightarrow X$
is classically known. Namely, we have the following generalization
of the main theorem of \cite{Snow88}: 
\begin{prop}
\label{prop:Rank2-bundle} Let $X$ be a n\oe therian normal scheme
defined over a field of characteristic $0$ and let $\pi:E\rightarrow X$
be a Zariski locally trivial $\mathbb{A}^{2}$-bundle equipped with
a fixed point free locally triangulable $\mathbb{G}_{a,X}$-action.
Then the geometric quotient $\mathrm{p}:E/\mathbb{G}_{a,X}\rightarrow X$
has the structure of a Zariski locally trivial $\mathbb{A}^{1}$-bundle
over $X$. \end{prop}
\begin{proof}
The assertion being local on the base $X$, we may assume that $X=\mathrm{Spec}(A)$
is the spectrum of a normal local domain containing a field of characteristic
$0$ and that $E=\mathbb{A}_{X}^{2}=\mathrm{Spec}(A[y,z])$ is equipped
with the $\mathbb{G}_{a,X}$-action generated by a triangular derivation
$\partial=a\partial_{y}+q(y)\partial_{z}$ of $A[y,z]$, where $a\in A$
and $q(y)\in A[y]$. The fixed point freeness hypothesis is equivalent
to the property that $a$ and $q(y)$ generate the unit ideal in $A[y,z]$.
So $q(y)$ has the form $q(y)=b+c\tilde{q}(y)$ where $b\in A$ is
relatively prime with $a$, $c\in\sqrt{aA}$ and $\tilde{q}(y)\in A[y]$.
Letting $Q(y)=\int_{0}^{y}q(\tau)d\tau=by+c\int_{0}^{y}\tilde{q}(\tau)d\tau$,
the polynomial $v=az-Q(y)\in A[y,z]$ belongs to the kernel $\mathrm{Ker}\partial$
of $\partial$ hence defines a $\mathbb{G}_{a,X}$-invariant morphism
$v:E\rightarrow\mathbb{A}_{X}^{1}=\mathrm{Spec}(A[t])$. Since $a$
and $b$ generate the unit ideal in $A$, it follows from the Jacobian
criterion that $v:E\rightarrow\mathbb{A}_{X}^{1}$ is a smooth morphism.
Furthermore, the fibers of $v$ coincide precisely with the $\mathbb{G}_{a,X}$-orbits
on $E$. Indeed, over the principal open subset $X_{a}=\mathrm{Spec}(A_{a})$
of $X$, $\partial$ admits $a^{-1}y$ as a slice and we have an equivariant
isomorphism $E\mid_{X_{a}}\simeq\mathrm{Spec}(A[a^{-1}v,a^{-1}y])\simeq\mathbb{A}_{X_{a}}^{1}\times_{X}\mathbb{G}_{a,X}$
where $\mathbb{G}_{a,X}$ acts by translations on the second factor.
On the other hand, the restriction $E\mid_{Z}$ of $E$ over the closed
subset $Z\subset X$ with defining ideal $\sqrt{aA}\subset A$ is
equivariantly isomorphic to $\mathbb{A}_{Z}^{2}$ equipped with the
$\mathbb{G}_{a,Z}$-action generated by the derivation $\overline{\partial}=\overline{b}\partial_{z}$
of $(A/\sqrt{aA})[y,z]$, where $\overline{b}\in(A/\sqrt{aA})^{*}$
denotes the residue class of $b$. The restriction of $v$ to $E\mid_{Z}$
coincides via this isomorphism to the morphism $\mathbb{A}_{Z}^{2}\rightarrow\mathbb{A}_{Z}^{1}$
defined by the polynomial $\overline{v}=\overline{b}y\in(A/\sqrt{aA})[y,z]$
which is obviously a geometric quotient. The above properties imply
that the morphism $\tilde{v}:E/\mathbb{G}_{a,X}\rightarrow\mathbb{A}_{X}^{1}$
induced by $v$ is smooth and bijective. Since it admits \'etale
quasi-sections, $\tilde{v}$ is then an isomorphism locally in the
\'etale topology on $\mathbb{A}_{X}^{1}$ whence an isomorphism. 
\end{proof}

\section{preliminary reductions }

\subsection{Reduction to a local base}

The statement of the Main Theorem can be rephrased equivalently as
the fact that a proper locally triangulable $\mathbb{G}_{a,S}$-action
on a Zariski locally trivial $\mathbb{A}^{3}$-bundle $\pi:E\rightarrow S$
is a translation in codimension $1$. This means that for every point
$s\in S$ of codimension $1$ with local ring $\mathcal{O}_{S,s}$,
the fiber product $E\times_{S}S'\simeq\mathbb{A}_{S'}^{3}$ of $E\rightarrow S$
with the canonical immersion $S'=\mathrm{Spec}(\mathcal{O}_{S,s})\hookrightarrow S$
equiped with the induced proper triangular action of $\mathbb{G}_{a,S'}=\mathbb{G}_{a,S}\times_{S}S'$
is equivariantly isomorphic to the trivial bundle $\mathbb{A}_{S'}^{2}\times_{S'}\mathbb{G}_{a,S'}$
over $S'$ equipped with the action of $\mathbb{G}_{a,S'}$ by translations
on the second factor. 

\begin{parn} \label{par:local_notation} So we are reduced to the
case where $S$ is the spectrum of a discrete valuation ring $A$
containing a field of characteristic $0$, say with maximal ideal
$\mathfrak{m}$ and residue field $\kappa=A/\mathfrak{m}$, and where
$\pi=\mathrm{pr}_{S}:E=\mathbb{A}_{S}^{3}=\mathrm{Spec}(A[y,z,u])\rightarrow S=\mathrm{Spec}(A)$
is equipped with a proper triangulable $\mathbb{G}_{a,S}$-action
$\sigma:\mathbb{G}_{a,S}\times_{S}\mathbb{A}_{S}^{3}\rightarrow\mathbb{A}_{S}^{3}$.
Letting $x\in\mathfrak{m}$ be uniformizing parameter, every such
action is equivalent to one generated by an $A$-derivation $\partial$
of $A[y,z,u]$ of the form 
\[
\partial=x^{n}\partial_{y}+q(y)\partial_{z}+p(y,z)\partial_{u}
\]
where $n\geq0$, $q(y)\in A[y]$ and $p(y,z)=\sum_{r=0}^{\ell}p_{r}(y)z^{r}\in A[y,z]$,
the fixed point freeness of $\sigma$ being equivalent to the property
that $x^{n}$, $q(y)$ and $p(y,z)$ generate the unit ideal in $A[y,z,u]$. 

\end{parn}

\subsection{\label{sub:Reduction-to-Affineness} Reduction to proving the affineness
of the geometric quotient}

With the notation of \S \ref{par:local_notation}, we can already
observe that if $n=0$ then $y$ is an obvious global slice for $\partial$
and hence that the action is globally equivariantly trivial with geometric
quotient $\mathfrak{X}=\mathbb{A}_{S}^{3}/\mathbb{G}_{a,S}\simeq\mathbb{A}_{S}^{2}$.
Similarly, if the residue class of $q(y)$ in $\kappa[y]$ is a non
zero constant then the action $\sigma$ is a translation. Indeed,
in this case, the $\mathbb{G}_{a,S}$-action $\overline{\sigma}:\mathbb{G}_{a,S}\times_{S}\mathbb{A}_{S}^{2}\rightarrow\mathbb{A}_{S}^{2}$
on $\mathbb{A}_{S}^{2}=\mathrm{Spec}(A[y,z])$ generated by the $A$-derivation
$\overline{\partial}=x^{n}\partial_{y}+q(y)\partial_{z}$ of $A[y,z]$
is fixed point free hence globally equivariantly trivial with geometric
quotient $\mathbb{A}_{S}^{2}/\mathbb{G}_{a,S}\simeq\mathbb{A}_{S}^{1}$
by virtue of Proposition \ref{prop:Rank2-bundle}. On the other hand,
the $\mathbb{G}_{a,S}$-equivariant projection $\mathrm{pr}_{y,z}:\mathbb{A}_{S}^{3}\rightarrow\mathbb{A}_{S}^{2}$
descends to a locally trivial $\mathbb{A}^{1}$-bundle between the
geometric quotients $\mathbb{A}_{S}^{3}/\mathbb{G}_{a,S}$ and $\mathbb{A}_{S}^{2}/\mathbb{G}_{a,S}$,
and since $\mathbb{A}_{S}^{2}/\mathbb{G}_{a,S}\simeq\mathbb{A}_{S}^{1}$
is affine and factorial, it follows that $\mathbb{A}_{S}^{3}/\mathbb{G}_{a,S}\simeq\mathbb{A}_{S}^{2}/\mathbb{G}_{a,S}\times_{S}\mathbb{A}_{S}^{1}\simeq\mathbb{A}_{S}^{2}.$
The affineness of $\mathbb{A}_{S}^{2}$ implies in turn that the quotient
morphism $\mathbb{A}_{S}^{3}\rightarrow\mathbb{A}_{S}^{3}/\mathbb{G}_{a,S}$
is the trivial $\mathbb{G}_{a,S}$-bundle whence that $\sigma:\mathbb{G}_{a,S}\times_{S}\mathbb{A}_{S}^{3}\rightarrow\mathbb{A}_{S}^{3}$
is a translation. Alternatively, one can observe that a global slice
$s\in A[y,z]$ for the action $\overline{\sigma}$ is also a global
slice for $\sigma$ via the inclusion $A[y,z]\subset A[y,z,u]$ 

More generally, the following Lemma reduces the question of global
equivariant triviality with geometric quotient $\mathfrak{X}=\mathbb{A}_{S}^{3}/\mathbb{G}_{a,S}$
isomorphic to $\mathbb{A}_{S}^{2}$ to showing that $\mathfrak{X}$,
which a priori only exists as an algebraic $S$-space, is an affine
$S$-scheme: 
\begin{lem}
\label{lem:Reduction_to_affineness} A fixed point free triangular
action $\sigma:\mathbb{G}_{a,S}\times_{S}\mathbb{A}_{S}^{3}\rightarrow\mathbb{A}_{S}^{3}$
is a translation if and only if its geometric quotient $\mathfrak{X}=\mathbb{A}_{S}^{3}/\mathbb{G}_{a,S}$
is an affine $S$-scheme. \end{lem}
\begin{proof}
One direction is clear, so assume that $\mathfrak{X}$ is an affine
$S$-scheme. It suffices to show that the structure morphism $\mathrm{p}:\mathfrak{X}\rightarrow S$
is an $\mathbb{A}^{2}$-fibration, i.e. a faithfully flat morphism
with all its fibers isomorphic to affine planes over the corresponding
residue fields. Indeed, if so, the affineness of $\mathfrak{X}$ implies
on the one hand that $\mathfrak{X}$ is isomorphic to the trivial
$\mathbb{A}^{2}$-bundle $\mathbb{A}_{S}^{2}$ by virtue of \cite{Sat83}
and on the other hand that $\rho:\mathbb{A}_{S}^{3}\rightarrow\mathfrak{X}$
is isomorphic to the trivial $\mathbb{G}_{a,S}$-bundle $\mathfrak{X}\times_{S}\mathbb{G}_{a,S}$
over $S$, which yields $\mathbb{G}_{a,S}$-equivariant isomorphisms
$\mathbb{A}_{S}^{3}\simeq\mathfrak{X}\times_{S}\mathbb{G}_{a,S}\simeq\mathbb{A}_{S}^{2}\times_{S}\mathbb{G}_{a,S}$. 

To see that $\mathrm{p}:\mathfrak{X}\rightarrow S$ is an $\mathbb{A}^{2}$-fibration,
recall that $\mathrm{pr}_{S}:\mathbb{A}_{S}^{3}\rightarrow S$ and
the quotient morphism $\rho:\mathbb{A}_{S}^{3}\rightarrow\mathfrak{X}=\mathbb{A}_{S}^{3}/\mathbb{G}_{a,S}$
are both faithfully flat, so that $\mathrm{p}:\mathfrak{X}\rightarrow S$
is faithfully flat too (\cite[II.3.2]{Knu71} and \cite[Corollaire 2.2.13(iii)]{EGA4}).
Letting $\mathfrak{m}$ and $\xi$ be the closed and generic points
of $S$ respectively, the fibers $\mathrm{pr}_{S}^{-1}(\mathfrak{m})\simeq\mathbb{A}_{\kappa}^{3}$
and $\mathrm{pr}_{S}^{-1}(\xi)\simeq\mathbb{A}_{\kappa(\xi)}^{3}$
coincide with the total spaces of the restriction of the $\mathbb{G}_{a,S}$-bundle
$\rho:\mathbb{A}_{S}^{3}\rightarrow\mathfrak{X}$ over the fibers
$\mathfrak{X}_{\mathfrak{m}}=\mathrm{p}^{-1}(\mathfrak{m})$ and $\mathfrak{X}_{\xi}=\mathrm{p}^{-1}(\xi)$
respectively. Since the $\mathbb{G}_{a,\kappa(\xi)}$-action induced
by $\sigma$ on $\mathrm{pr}_{S}^{-1}(\xi)$ admits $x^{-n}y$ as
a global slice, it is a translation with geometric quotient $\mathbb{A}_{\kappa(\xi)}^{3}/\mathbb{G}_{a,\kappa(\xi)}\simeq\mathbb{A}_{\kappa(\xi)}^{2}$
and so $\mathfrak{X}_{\xi}\simeq\mathbb{A}_{\kappa(\xi)}^{2}$. On
the other hand, we may assume in view of the above discussion that
$n\geq1$ so that the $\mathbb{G}_{a,\kappa}$-action on $\mathrm{pr}_{S}^{-1}(\mathfrak{m})\simeq\mathbb{A}_{\kappa}^{3}$
induced by $\sigma$ coincides with the fixed point free action generated
by the $\kappa[y]$-derivation $\overline{\partial}=\overline{q}(y)\partial_{z}+\overline{p}(y,z)\partial_{u}$
of $\kappa[y][z,u]$, where $\overline{q}(y)$ and $\overline{p}(y,z)$
denote the respective residue classes of $q(y)$ and $p(y,z)$ modulo
$x$. By virtue of Proposition \ref{prop:Rank2-bundle}, the geometric
quotient $\mathbb{A}_{\kappa}^{3}/\mathbb{G}_{a,\kappa}$ has the
structure of a Zariski locally trivial $\mathbb{A}^{1}$-bundle over
$\mathbb{A}_{\kappa}^{1}=\mathrm{Spec}(\kappa[y])$ hence is isomorphic
to $\mathbb{A}_{\kappa}^{2}$. This implies that $\mathfrak{X}_{\mathfrak{m}}\simeq\mathbb{A}_{\kappa}^{3}/\mathbb{G}_{a,\kappa}\simeq\mathbb{A}_{\kappa}^{2}$
as desired. \end{proof}
\begin{rem}
By exploiting the fact that arbitrary $\mathbb{G}_{a,S}$-actions
on the affine $3$-space $\mathbb{A}_{S}^{3}$ over the spectrum $S$
of a discrete valuation ring $A$ containing a field of characteristic
$0$ have finitely generated rings of invariants \cite{BhaDa09},
one can derive the following stronger characterization: a fixed point
free action $\sigma:\mathbb{G}_{a,S}\times_{S}\mathbb{A}_{S}^{3}\rightarrow\mathbb{A}_{S}^{3}$
is either a translation or its geometric quotient $\mathfrak{X}=\mathbb{A}_{S}^{3}/\mathbb{G}_{a,S}$
is an algebraic space which is not a scheme. 

Indeed, the quotient morphism $\rho:\mathbb{A}_{S}^{3}\rightarrow\mathfrak{X}$
is again an $\mathbb{A}^{2}$-fibration thanks to \cite[Theorem 3.2]{DaiKal09}
which asserts that for every field $\kappa$ of characteristic $0$
a fixed point free action of $\mathbb{G}_{a,\kappa}$-action on $\mathbb{A}_{\kappa}^{3}$
is a translation, and so the assertion is equivalent to the fact that
a Zariski locally equivariantly trivial action $\sigma$ has affine
geometric quotient $\mathfrak{X}$. This can be seen in a similar
way as in the proof of Theorem 2.1 in \cite{DevFinvR04}. Namely,
by hypothesis we can find an open covering of $\mathbb{A}_{S}^{3}$
by finitely many invariant affine open subsets $U_{i}$ on which the
induced $\mathbb{G}_{a,S}$-action is a translation with affine geometric
quotient $U_{i}/\mathbb{G}_{a,S}$, $i=1,\ldots,n$. Since $U_{i}$
and $\mathbb{A}_{S}^{3}$ are affine, $\mathbb{A}_{S}^{3}\setminus U_{i}$
is a $\mathbb{G}_{a,S}$-invariant Weil divisor on $\mathbb{A}_{S}^{3}$
which is in fact principal as $A$, whence $A[y,z,u]$, is factorial.
It follows that there exists invariant regular functions $f_{i}\in A[y,z,u]^{\mathbb{G}_{a}}\simeq\Gamma(\mathfrak{X},\mathcal{O}_{\mathfrak{X}})$
such that $U_{i}=\mathrm{Spec}(A[x,y,z]_{f_{i}})$ coincides with
the inverse image by the quotient morphism $\rho:\mathbb{A}_{S}^{3}\rightarrow\mathfrak{X}$
of the principal open subset $\mathfrak{X}_{f_{i}}$ of $\mathfrak{X}$,
$i=1,\ldots,n$. Since $\rho:\mathbb{A}_{S}^{3}\rightarrow\mathfrak{X}$
is a $\mathbb{G}_{a,S}$-bundle and $U_{i}\simeq U_{i}/\mathbb{G}_{a,S}\times_{S}\mathbb{G}_{a,S}$
by assumption, we conclude that $\mathfrak{X}$ is covered by the
principal affine open subsets $\mathfrak{X}_{f_{i}}\simeq U_{i}/\mathbb{G}_{a,S}$,
$i=1,\ldots,n$, whence is quasi-affine. Now since by the aforementioned
result \cite{BhaDa09}, $A[y,z,u]^{\mathbb{G}_{a}}$ is an integrally
closed finitely generated $A$-algebra, it is enough to check that
the canonical open immersion $j:\mathfrak{X}\rightarrow X=\mathrm{Spec}(\Gamma(\mathfrak{X},\mathcal{O}_{\mathfrak{X}}))\simeq\mathrm{Spec}(A[y,z,u]^{\mathbb{G}_{a}})$
is surjective. The surjectivity over the generic point of $S$ follows
immediately from the fact the kernel of a locally nilpotent derivation
derivation of a polynomial ring in three variables over a field $K$
of characteristic $0$ is isomorphic to a polynomial ring in two variables
over $K$ (see e.g. \cite{Miy85}). So it remains to show that the
induced open immersion $j_{\mathfrak{m}}:\mathfrak{X}_{m}\simeq\mathbb{A}_{\kappa}^{2}\hookrightarrow X_{\mathfrak{m}}=\mathrm{Spec}(A[y,z,u]^{\mathbb{G}_{a}}\otimes_{A}A/\mathfrak{m})$
between the corresponding fibers over the closed point $\mathfrak{m}$
of $S$ is surjective, in fact, an isomorphism. Since $x\in A[y,z,u]^{\mathbb{G}_{a}}$
is prime, $X_{\mathfrak{m}}\simeq\mathrm{Spec}(A[y,z,u]^{\mathbb{G}_{a}}/(x))$
is an integral $\kappa$-scheme of finite type and Corollary 4.10
in \cite{BhaDa09} can be interpreted more precisely as the fact that
$X_{\mathfrak{m}}\simeq C\times_{\kappa}\mathbb{A}_{\kappa}^{1}$
for a certain $1$-dimensional affine $\kappa$-scheme $C$. This
implies in turn that $j_{\mathfrak{m}}$ is an isomorphism. Indeed,
since $C$ is dominated via $j_{\mathfrak{m}}$ by a general affine
line $\mathbb{A}_{\kappa}^{1}\subset\mathbb{A}_{\kappa}^{2}$, its
normalization $\tilde{C}$ is isomorphic to $\mathbb{A}_{\kappa}^{1}$
and so $j_{\mathfrak{m}}$ factors through an open immersion $\tilde{j}_{\mathfrak{m}}:\mathbb{A}_{\kappa}^{2}\hookrightarrow\tilde{C}\times_{\kappa}\mathbb{A}_{\kappa}^{1}\simeq\mathbb{A}_{\kappa}^{2}$.
The latter is surjective for otherwise the complement of its image
would be of pure codimension $1$ hence a principal divisor $\mathrm{div}(f)$
for a non constant regular function $f$ on $\tilde{C}\times_{\kappa}\mathbb{A}_{\kappa}^{1}$.
But then $f$ would restrict to a non constant invertible function
on the image of $\mathbb{A}_{\kappa}^{2}$ which is absurd. Thus $\tilde{j}_{\mathfrak{m}}:\mathbb{A}_{\kappa}^{2}\hookrightarrow\tilde{C}\times_{\kappa}\mathbb{A}_{\kappa}^{1}\simeq\mathbb{A}_{\kappa}^{2}$
is an isomorphism and since the normalization morphism $\tilde{C}\times_{\kappa}\mathbb{A}_{\kappa}^{1}\rightarrow C\times_{\kappa}\mathbb{A}_{\kappa}^{1}$
is finite whence closed it follows that $j_{\mathfrak{m}}:\mathbb{A}_{\kappa}^{2}\hookrightarrow C\times_{\kappa}\mathbb{A}_{\kappa}^{1}$
is an open and closed immersion hence an isomorphism.
\end{rem}

\subsection{Reduction to extensions of irreducible derivations}

In view of the discussion at the beginning of subsection \ref{sub:Reduction-to-Affineness},
we may assume for the $A$-derivation 
\[
\partial=x^{n}\partial_{y}+q(y)\partial_{z}+p(y,z)\partial_{u}
\]
that $n>0$ and that the residue class of $q(y)$ in $\kappa[y]$
is either zero or not constant. In the first case, $q(y)\in\mathfrak{m}A[y]$
has the form $q(y)=x^{\mu}q_{0}(y)$ where $\mu>0$ and where $q_{0}(y)\in A[y]$
has non zero residue class modulo $\mathfrak{m}$, so that the derivation
$\overline{\partial}=x^{n}\partial_{y}+q(y)\partial_{z}$ induced
by $\partial$ on the sub-ring $A[y,z]$ is reducible. On the other
hand, the fixed point freeness of the $\mathbb{G}_{a,S}$-action $\sigma$
generated by $\partial$ implies that up to multiplying $u$ by an
invertible element in $A$, one has $p(y,z)=1+x^{\nu}p_{0}(y,z)$
for some $\nu>0$ and $p_{0}(y,z)\in A[y,z]$. 

If $\mu\geq n$, then letting $Q_{0}(y)=\int_{0}^{y}q_{0}(\tau)d\tau\in A[y]$,
the $\mathbb{G}_{a,S}$-invariant polynomial $z_{1}=z-x^{\mu-n}Q_{0}(y)$
is a variable of $A[y,z,u]$ over $A[y,u]$, and so $\partial$ is
conjugate to the derivation $x^{n}\partial_{y}+p(y,z_{1}+x^{\mu-n}Q_{0}(y))\partial_{u}$
of the polynomial ring in two variables $A[z_{1}][y,u]$ over $A[z_{1}]$.
Since $\sigma$ is fixed point free, Proposition \ref{prop:Rank2-bundle}
implies that it is equivariantly trivial with geometric quotient isomorphic
to the total space of the trivial $\mathbb{A}^{1}$-bundle over $\mathbb{A}_{S}^{1}=\mathrm{Spec}(A[z_{1}])$
whence to $\mathbb{A}_{S}^{2}$. 

Otherwise, if $\mu<n$, then the $\mathbb{G}_{a,S}$-action $\tilde{\sigma}:\mathbb{G}_{a,S}\times_{S}\mathbb{A}_{S}^{3}\rightarrow\mathbb{A}_{S}^{3}$
on $\mathbb{A}_{S}^{3}=\mathrm{Spec}(A[\tilde{y},\tilde{z},\tilde{u}])$
generated by the $A$-derivation 
\[
\tilde{\partial}=x^{n-\mu}\partial_{\tilde{y}}+q_{0}(\tilde{y})\partial_{\tilde{z}}+(1+x^{\nu}p_{0}(\tilde{y},\tilde{z}))\partial_{\tilde{u}}
\]
is again fixed point free, hence admits a geometric quotient $\tilde{\rho}:\mathbb{A}_{S}^{3}\rightarrow\tilde{\mathfrak{X}}=\mathbb{A}_{S}^{3}/\mathbb{G}_{a,S}$
in the form of an \'etale locally trivial $\mathbb{G}_{a,S}$-bundle
over a certain algebraic $S$-space $\tilde{\mathfrak{X}}$. 
\begin{lem}
The quotient spaces $\mathfrak{X}=\mathbb{A}_{S}^{3}/\mathbb{G}_{a,S}$
and $\tilde{\mathfrak{X}}=\mathbb{A}_{S}^{3}/\mathbb{G}_{a,S}$ for
the $\mathbb{G}_{a,S}$-actions $\sigma$ and $\tilde{\sigma}$ on
$\mathbb{A}_{S}^{3}$ generated by $\partial$ and $\tilde{\partial}$
respectively are isomorphic. In particular $\sigma$ is proper (resp.
equivariantly trivial) if and only if $\tilde{\sigma}$ is proper
(resp. equivariantly trivial). \end{lem}
\begin{proof}
Letting $\tilde{\rho}_{i}:V_{i}=\mathbb{A}_{S}^{3}\rightarrow\tilde{\mathfrak{X}}_{i}=V_{i}/\mathbb{G}_{a,S}$,
$i=0,\ldots,\mu$, denote the geometric quotient of $V_{i}=\mathrm{Spec}(A[\tilde{y}_{i},\tilde{z}_{i},\tilde{u}_{i}])$
for the fixed point free $\mathbb{G}_{a,S}$-action $\tilde{\sigma}_{i}$
generated by the $A$-derivation 
\[
\tilde{\partial}_{i}=\left(1+x^{\nu}p_{0}(\tilde{y}_{i},\tilde{z}_{i})\right)\partial_{\tilde{u}_{i}}+x^{\mu-i}q_{0}(\tilde{y}_{i})\partial_{\tilde{z}_{i}}+x^{n-i}\partial_{\tilde{y}_{i}},
\]
the first assertion will follow from the more general fact that $\tilde{\mathfrak{X}}_{i}\simeq\tilde{\mathfrak{X}}_{i+1}$
for every $i=0,\ldots,\mu-1$. Indeed, we first observe that since
$\tilde{u}_{i}$ is a slice for $\tilde{\partial}_{i}$ modulo $x$,
$\tilde{\mathfrak{X}}_{i,\mathfrak{m}}=\tilde{\mathfrak{X}}_{i}\times_{S}\mathrm{Spec}(\kappa)$
is isomorphic to $\mathbb{A}_{\kappa}^{2}=\mathrm{Spec}((A/\mathfrak{m})[\tilde{y}_{i},\tilde{z}_{i}])$
and the restriction of $\tilde{\rho}_{i}$ over $\tilde{\mathfrak{X}}_{i,\mathfrak{m}}$
is isomorphic to the trivial bundle $\mathrm{pr}_{1}:\tilde{\mathfrak{X}}_{i,\mathfrak{m}}\times_{\kappa}\mathrm{Spec}(\kappa[\tilde{u}_{i}])\rightarrow\tilde{\mathfrak{X}}_{i,\mathfrak{m}}$.
Now let $\beta_{i}:V_{i+1}\rightarrow V_{i}$ be the affine modification
of the total space of $\tilde{\rho}_{i}:\mathbb{A}_{S}^{3}\rightarrow\tilde{\mathfrak{X}}_{i}$
with center at the zero section of the induced bundle $\mathrm{pr}_{1}:\tilde{\mathfrak{X}}_{i,\mathfrak{m}}\times_{\kappa}\mathrm{Spec}(\kappa[\tilde{u}_{i}])\rightarrow\tilde{\mathfrak{X}}_{i,\mathfrak{m}}$
and with principal divisor $x$. In view of the previous description,
$\beta_{i}:V_{i+1}\rightarrow V_{i}$ coincides with the affine modification
of $\mathrm{Spec}(A[\tilde{y}_{i},\tilde{z}_{i},\tilde{u}_{i}])$
with center at the ideal $(x,\tilde{u}_{i})$ and principal divisor
$x$, that is, with the birational $S$-morphism induced by the homomorphism
of $A$-algebra 
\[
\beta_{i}^{*}:A[\tilde{y}_{i+1},\tilde{z}_{i+1},\tilde{u}_{i+1}]\rightarrow A[\tilde{y}_{i},\tilde{z}_{i},\tilde{u}_{i}],\;(\tilde{y}_{i+1},\tilde{z}_{i+1},\tilde{u}_{i+1})\mapsto(\tilde{y}_{i},\tilde{z}_{i},x\tilde{u}_{i}).
\]
By construction, $\beta_{i}$ is equivariant for the $\mathbb{G}_{a,S}$-actions
$\tilde{\sigma}_{i+1}$ and $\overline{\sigma}_{i}$ generated respectively
by the locally nilpotent $A$-derivations $\tilde{\partial}_{i+1}$
of $A[\tilde{y}_{i+1},\tilde{z}_{i+1},\tilde{u}_{i+1}]$ and $\overline{\partial}_{i}=x\tilde{\partial}_{i}$
of $A[\tilde{y}_{i},\tilde{z}_{i},\tilde{u}_{i}]$. Furthermore, since
$\tilde{\rho}_{i}:V_{i}\rightarrow\tilde{\mathfrak{X}}_{i}$ is also
$\mathbb{G}_{a,S}$-invariant for the action $\overline{\sigma}_{i}$,
the morphism $\tilde{\rho}_{i}\circ\beta_{i}:V_{i+1}\rightarrow\tilde{\mathfrak{X}}_{i}$
is $\mathbb{G}_{a,S}$-invariant, whence descends to a morphism $\tilde{\beta}_{i}:\tilde{\mathfrak{X}}_{i+1}\rightarrow\tilde{\mathfrak{X}}_{i}$.
Since the latter restricts to an isomorphism over the generic point
of $S$, it remains to check that it is also an isomorphism in a neighborhood
of every point $p\in\tilde{\mathfrak{X}}_{i}$ lying over the closed
point $\mathfrak{m}$ of $S$. Let $f:U=\mathrm{Spec}(B)\rightarrow\tilde{\mathfrak{X}}_{i}$
be an affine \'etale neighborhood of such a point $p\in\tilde{\mathfrak{X}}_{i}$
over which $\tilde{\rho}_{i}:V_{i}\rightarrow\tilde{\mathfrak{X}}_{i}$
becomes trivial, say $V_{i}\times_{\tilde{\mathfrak{X}}_{i}}U\simeq\mathbb{A}_{U}^{1}=\mathrm{Spec}(B[\tilde{v}_{i}])$.
The $\mathbb{G}_{a,S}$-action on $V_{i}$ generated by $\overline{\partial}_{i}$
lifts to the $\mathbb{G}_{a,U}$-action on $\mathbb{A}_{U}^{1}$ generated
by the locally nilpotent $B$-derivation $x\partial_{\tilde{v}_{i}}$
and since $\beta_{i}:V_{i+1}\rightarrow V_{i}$ is the affine modification
of $V_{i}$ with center at the zero section of the restriction of
$\tilde{\rho}_{i}:V_{i}\rightarrow\tilde{\mathfrak{X}}_{i}$ over
the closed point of $S$, we have a commutative diagram \[\xymatrix@R=12pt@C=11pt{ & V_{i+1} \ar'[d][dd]_(.3){\tilde{\rho}_{i+1}} \ar[dl]_{\beta_i} && \mathbb{A}^1_U \ar[ll] \ar[dd]^{\mathrm{pr}_U} \ar[dl]_{\delta_i} \\ V_i \ar[dd]_{\tilde{\rho}_i} && \mathbb{A}^1_U \ar[ll] \ar[dd]^(.3){\mathrm{pr}_U} \\ & \tilde{\mathfrak{X}}_{i+1} \ar[dl]_{\tilde{\beta}_i} && U \ar@{=}[dl] \ar'[l][ll] \\ \tilde{\mathfrak{X}}_i && U \ar[ll]_{f} }\]in
which the top and front squares are cartesian, and where the morphism
$\delta_{i}:\mathbb{A}_{U}^{1}=\mathrm{Spec}(B[\tilde{v}_{i+1}])\rightarrow\mathbb{A}_{U}^{1}=\mathrm{Spec}(B[\tilde{v}_{i}])$
is defined by the $B$-algebras homomorphism $B[\tilde{v}_{i}]\rightarrow B[\tilde{v}_{i+1}]$,
$\tilde{v}_{i}\mapsto x\tilde{v}_{i+1}$. The latter is equivariant
for the action on $\mathrm{Spec}(B[\tilde{v}_{i+1}])$ generated by
the locally nilpotent $B$-derivation $\partial_{\tilde{v}_{i+1}}$
and we conclude that $\mathrm{pr}_{2}:\mathbb{A}_{U}^{1}\simeq\mathbb{A}_{U}^{1}\times_{V_{i}}V_{i+1}\rightarrow V_{i+1}$
is an \'etale trivialization of the $\mathbb{G}_{a,S}$-action induced
by $\tilde{\sigma}_{i+1}$ on the open sub-scheme $(\tilde{\rho}_{i}\circ\beta_{i})^{-1}(f(U))$
of $V_{i+1}$. This implies in turn that $U\times_{\tilde{\mathfrak{X}}_{i}}\tilde{\mathfrak{X}}_{i+1}\simeq U$,
whence that $\tilde{\beta}_{i}:\tilde{\mathfrak{X}}_{i+1}\rightarrow\tilde{\mathfrak{X}}_{i}$
is an isomorphism in a neighborhood of $p\in\tilde{\mathfrak{X}}_{i}$
as desired. 

The second assertion is a direct consequence of the fact that properness
and global equivariant triviality of $\sigma$ and $\tilde{\sigma}$
are respectively equivalent to the separatedness and the affineness
of the geometric quotients $\mathfrak{X}\simeq\tilde{\mathfrak{X}}$.
\end{proof}
\begin{parn} Summing up, we are now reduced to proving that a proper
$\mathbb{G}_{a,S}$-action on $\mathbb{A}_{S}^{3}$ generated by an
$A$-derivation 
\[
\partial=x^{n}\partial_{y}+q(y)\partial_{z}+p(y,z)\partial_{u}
\]
of $A[y,z,u]$, such that $n>0$ and $q(y)\in A[y]$ has non constant
residue class in $\kappa[y]$, has affine geometric quotient $\mathfrak{X}=\mathbb{A}_{S}^{3}/\mathbb{G}_{a,S}$.
This will be done in two steps in the next sections: we will first
establish that a proper $\mathbb{G}_{a,S}$-action as above is conjugate
to one generated by a special type of $A$-derivation called \emph{twin-triangular.}
Then we will prove in section \ref{sec:Twin-Triviality} that proper
twin-triangular $\mathbb{G}_{a,S}$-actions on $\mathbb{A}_{S}^{3}$
do indeed have affine geometric quotients. 

\end{parn}

\section{Reduction to twin-triangular actions}

We keep the same notation as in \S \ref{par:local_notation} above,
namely $A$ is a discrete valuation ring containing a field of characteristic
$0$, with maximal ideal $\mathfrak{m}$, residue field $\kappa=A/\mathfrak{m}$,
and uniformizing parameter $x\in\mathfrak{m}$. We let again $S=\mathrm{Spec}(A)$.

We call an $A$-derivation $\partial$ of $A[y,z,u]$ \emph{twin-triangulable}
if there exists a coordinate system $(y,z_{+},z_{-})$ of $A[y,z,u]$
over $A[y]$ in which the conjugate of $\partial$ is \emph{twin-triangular},
that is, has the form $x^{n}\partial_{y}+p_{+}(y)\partial_{z_{+}}+p_{-}(y)\partial_{z_{-}}$
for certain polynomials $p_{\pm}(y)\in A[y]$. This section is devoted
to the proof of the following intermediate characterization of proper
triangular $\mathbb{G}_{a,S}$-actions: 
\begin{prop}
\label{prop:Main-Local} With the notation above, let $\partial$
by an $A$-derivation of $A[y,z,u]$ of the form 
\[
\partial=x^{n}\partial_{y}+q(y)\partial_{z}+p(y,z)\partial_{u}
\]
where $n>0$ and where $q(y)\in A[y]$ has non constant residue class
in $\kappa[y]$. If the $\mathbb{G}_{a,S}$-action on $\mathbb{A}_{S}^{3}=\mathrm{Spec}(A[y,z,u])$
generated by $\partial$ is proper, then $\partial$ is twin-triangulable.
\end{prop}
\noindent The proof given below proceeds in two steps: we first construct
a coordinate $\tilde{u}$ of $A[y,z,u]$ over $A[y,z]$ with the property
that $\partial\tilde{u}=\tilde{p}(y,z)$ is either a polynomial in
$y$ only or its leading term $\tilde{p}_{\ell}(y)$ as a polynomial
in $z$ has a very particular form. In the second case, we exploit
the properties of $\tilde{p}_{\ell}(y)$ to show that the $\mathbb{G}_{a,S}$-action
generated by $\partial$ is not proper.

\subsection{The $\sharp$-reduction of a triangular $A$-derivation}

The conjugate of an $A$-derivation $\partial=x^{n}\partial_{y}+q(y)\partial_{z}+p(y,z)\partial_{u}$
of $A[y,z,u]$ as in Proposition \ref{prop:Main-Local} by an isomorphism
of $A[y,z]$-algebras $\psi:A[y,z][\tilde{u}]\stackrel{\sim}{\rightarrow}A[y,z][u]$
is again triangular of the form 
\[
\psi^{-1}\partial\psi=x^{n}\partial_{y}+q(y)\partial_{z}+\tilde{p}(y,z)\partial_{\tilde{u}}
\]
for some polynomial $\tilde{p}(y,z)\in A[y,z]$. In particular, we
may choose from the very beginning a coordinate system of $A[y,z,u]$
over $A[y,z]$ with the property that the degree of $\partial u\in A[y,z]$
with respect to $z$ is minimal among all possible conjugates $\psi^{-1}\partial\psi$
of $\partial$ as above. In what follows, we will say for short that
such a derivation $\partial$ is \emph{$\sharp$-reduced} with respect
to the coordinate system $(y,z,u)$. Letting $Q(y)=\int_{0}^{y}q(\tau)d\tau\in A[y]$,
this property can be characterized effectively as follows: 
\begin{lem}
Let $\partial=x^{n}\partial_{y}+q(y)\partial_{z}+p(y,z)\partial_{u}$
be a $\sharp$-reduced derivation of $A[y,z,u]$ as in Proposition
\ref{prop:Main-Local}. If $\partial$ is not twin-triangular $($i.e.
$p(y,z)=p_{0}(y)\in A[y]$$)$ then the leading term $p_{\ell}(y)$,
$\ell\geq1$, of $p(y,z)$ as a polynomial in $z$ is not congruent
modulo $x^{n}$ to a polynomial of the form $q(y)f(Q(y))$ for some
$f(\tau)\in A[\tau]$. \end{lem}
\begin{proof}
Suppose that $p(y,z)=\sum_{r=0}^{\ell}p_{r}(y)z^{r}$ with $\ell\geq1$
and that $p_{\ell}(y)=q(y)f(Q(y))+x^{n}g(y)$ for some polynomials
$f(\tau),g(\tau)\in A[\tau]$. Then letting $G(y)=\int_{0}^{y}g(\tau)d\tau$
and 
\[
\tilde{u}=u-G(y)z^{\ell}-\sum_{k=0}^{\deg f}\frac{(-1)^{k}}{\prod_{j=0}^{k}(\ell+1+j)}f^{(k)}(Q(y))x^{kn}z^{\ell+1+k},
\]
one checks by direct computation that 
\[
\partial\tilde{u}=\sum_{r=0}^{\ell-2}p_{r}(y)z^{r}+\left(p_{\ell-1}(y)-G(y)q(y)\right)z^{\ell-1}.
\]
Thus $(y,z,\tilde{u}$) is a coordinate system of $A[y,z,u]$ over
$A[y,z]$ in which the image of $\tilde{u}$ by the conjugate of $\partial$
has degree $\leq\ell-1$, a contradiction to the $\sharp$-reducedness
of $\partial$. 
\end{proof}
\noindent To prove Proposition \ref{prop:Main-Local}, it remains
to show that a proper $\mathbb{G}_{a,S}$-action on $\mathbb{A}_{S}^{3}$
generated by $\sharp$-reduced $A$-derivation of $A[y,z,u]$ is twin-triangular.
This is done in the next sub-section.

\subsection{A non-valuative criterion for non-properness}

\indent\newline\noindent To disprove the properness of an algebraic
action $\sigma:\mathbb{G}_{a,S}\times_{S}E\rightarrow E$ of $\mathbb{G}_{a,S}$
on an $S$-scheme $E$, it suffices in principle to check that the
image of $\Phi=(\mathrm{pr}_{2},\sigma):\mathbb{G}_{a}\times_{S}E\rightarrow E\times_{S}E$
is not closed. However, this image turns out to be complicated to
determine in general, and it is more convenient for our purpose to
consider the following auxiliary construction: letting $j:\mathbb{G}_{a,S}\simeq\mathrm{Spec}(\mathcal{O}_{S}[t])\hookrightarrow\mathbb{P}_{S}^{1}=\mathrm{Proj}(\mathcal{O}_{S}[w_{0},w_{1}])$,
$t\mapsto[t:1]$ be the natural open immersion, the properness of
the projection $\mathrm{pr}_{E\times_{S}E}:\mathbb{P}_{S}^{1}\times_{S}E\times_{S}E\rightarrow E\times_{S}E$
implies that $(\mathrm{p}_{2},\sigma)$ is proper if and only if $\varphi=(j\circ\mathrm{pr}_{1},\mathrm{pr}_{2},\sigma):\mathbb{G}_{a,S}\times_{S}E\rightarrow\mathbb{P}_{S}^{1}\times_{S}E\times_{S}E$
is proper, hence a closed immersion. Therefore the non properness
of $\sigma$ is equivalent to the fact that the closure of $\mathrm{Im}(\varphi)$
in $\mathbb{P}_{S}^{1}\times_{S}E\times_{S}E$ intersects the ``boundary''
$\{w_{1}=0\}$ in a nontrivial way.

\begin{parn} Now let $\sigma:\mathbb{G}_{a,S}\times_{S}\mathbb{A}_{S}^{3}\rightarrow\mathbb{A}_{S}^{3}$
be the $\mathbb{G}_{a,S}$-action generated by a non twin-triangular
$\sharp$-reduced $A$-derivation $\partial=x^{n}\partial_{y}+q(y)\partial_{z}+p(y,z)\partial_{u}$
of $A[y,z,u]$ and let 
\[
\varphi=(j\circ\mathrm{pr}_{1},\mathrm{pr}_{2},\mu):\mathbb{G}_{a,S}\times_{S}\mathbb{A}_{S}^{3}=\mathrm{Spec}(A[t][y,z,u])\rightarrow\mathbb{P}_{S}^{1}\times_{S}\mathbb{A}_{S}^{3}\times_{S}\mathbb{A}_{S}^{3}
\]
be the corresponding immersion. To disprove the properness of $\sigma$,
it is enough to check that the image by $\varphi$ of the closed sub-scheme
$H=\left\{ z=0\right\} \simeq\mathrm{Spec}(A[t][y,u])$ of $\mathbb{G}_{a,S}\times_{S}\mathbb{A}_{S}^{3}$
is not closed in $\mathbb{P}_{S}^{1}\times_{S}\mathbb{A}_{S}^{3}\times_{S}\mathbb{A}_{S}^{3}$.
After identifying $A[y,z,u]\otimes_{A}A[y,z,u]$ with the polynomial
ring $A[y_{1},y_{2},z_{1},z_{2},u_{1},u_{2}]$ in the obvious way,
the image of $H$ by $(\mathrm{pr}_{1},\mathrm{pr}_{2},\sigma):\mathbb{G}_{a,S}\times_{S}\mathbb{A}_{S}^{3}\rightarrow\mathbb{A}_{S}^{1}\times_{S}\mathbb{A}_{S}^{3}\times_{S}\mathbb{A}_{S}^{3}$
is equal to the closed sub-scheme of $\mathrm{Spec}(A[t][y_{1},y_{2},z_{1},z_{2},u_{1},u_{2}])$
defined by the following system of equations 
\[
\begin{cases}
y_{2} & =y_{1}+x^{n}t\\
z_{1} & =0\\
z_{2} & =x^{-n}(Q(y_{1}+x^{n}t)-Q(y_{1}))=(y_{1}-y_{2})^{-1}(Q(y_{2})-Q(y_{1}))t\\
u_{2} & =u_{1}+x^{-n}\int_{0}^{t}p(y_{1}+x^{n}\tau)(Q(y_{1}+x^{n}\tau)-Q(y_{1})))d\tau.
\end{cases}
\]
Letting $p(y,z)=\sum_{r=0}^{\ell}p_{r}(y)z^{r}$ with $\ell\geq1$
and 
\[
\Gamma_{r}(y_{1},y_{2})=\int_{y_{1}}^{y_{2}}p_{r}(\xi)(Q(\xi)-Q(y_{1}))^{r}d\xi\in A[y_{1},y_{2}],\quad r=0,\ldots,\ell,
\]
the last equality can be re-written modulo the first ones in the form
\begin{eqnarray*}
u_{2} & = & u_{1}+\sum_{r=0}^{\ell}x^{-nr}\int_{0}^{t}p_{r}(y_{1}+x^{n}\tau)(Q(y_{1}+x^{n}\tau)-Q(y_{1}))^{r}d\tau\\
 & = & u_{1}+t(y_{2}-y_{1})^{-1}\sum_{r=0}^{\ell}x^{-nr}\int_{y_{1}}^{y_{2}}p_{r}(\xi)(Q(\xi)-Q(y_{1}))^{r}d\xi\\
 & = & u_{1}+\sum_{r=0}^{\ell}\left((y_{2}-y_{1})^{-r-1}\Gamma_{r}(y_{1},y_{2})\right)t^{r+1}.
\end{eqnarray*}
It follows that the closure $V$ of $\varphi(H)$ is contained in
the closed sub-scheme $W$ of $\mathbb{P}_{S}^{1}\times_{S}\mathbb{A}_{S}^{3}\times_{S}\mathbb{A}_{S}^{3}$
defined by the equations $z_{1}=0$ and 
\[
\begin{cases}
(y_{2}-y_{1})w_{1}-x^{n}w_{0} & =0\\
w_{1}z_{2}-(y_{2}-y_{1})^{-1}(Q(y_{2})-Q(y_{1}))w_{0} & =0\\
w_{1}^{\ell+1}(u_{2}-u_{1})-\sum_{r=0}^{\ell}\left((y_{2}-y_{1})^{-r-1}\Gamma_{r}(y_{1},y_{2})\right)w_{0}^{r+1}w_{1}^{\ell-r} & =0.
\end{cases}
\]
We further observe that $W$ is irreducible, whence equal to $V$,
provided that $\Gamma_{\ell}(y_{1},y_{2})\in A[y_{1},y_{2}]$ does
not belong to the ideal generated by $x^{n}$ and $Q(y_{2})-Q(y_{1})$.
If so, then $W=V$ intersects $\{w_{1}=0\}$ along a closed sub-scheme
$Z$ isomorphic to the spectrum of the following algebra:

\[
\left(A[y_{1},y_{2}]/(x^{n},(y_{2}-y_{1})^{-1}(Q(y_{2})-Q(y_{1})),(y_{2}-y_{1})^{-\ell-1}\Gamma_{\ell}(y_{1},y_{2}))\right)[z_{2},u_{1},u_{2}].
\]
Since by virtue of the $\sharp$-reducedness assumption $p_{\ell}(y)$
is not of the form $q(y)f(Q(y))+x^{n}g(y)$, the non properness of
$\sigma:\mathbb{G}_{a,S}\times_{S}\mathbb{A}_{S}^{3}\rightarrow\mathbb{A}_{S}^{3}$
is then a consequence of the following Lemma which guarantees precisely
that $\Gamma_{\ell}(y_{1},y_{2})\not\in(x^{n},Q(y_{2})-Q(y_{1}))A[y_{1},y_{2}]$
and that $Z$ is not empty. 

\end{parn} 
\begin{lem}
Let $q(y)\in A[y]$ be a polynomial with non constant residue class
in $\kappa[y]$ and let $Q(y)=\int_{0}^{y}q(\tau)d\tau$. For a polynomial
$p(y)\in A[y]$ and an integer $\ell\geq1$, the following holds:

a) The polynomial $\Gamma_{\ell}(y_{1},y_{2})=\int_{y_{1}}^{y_{2}}p(y)(Q(y)-Q(y_{1}))^{\ell}dy$
belongs to the ideal $(x^{n},Q(y_{2})-Q(y_{1}))$ if and only if $p(y)$
can be written in the form $q(y)f(Q(y))+x^{n}g(y)$ for certain polynomials
$f(\tau),g(\tau)\in A[\tau]$. 

b) The polynomial $(y_{2}-y_{1})^{-\ell-1}\Gamma_{\ell}(y_{1},y_{2})$
is not invertible modulo the ideal $(x^{n},(y_{2}-y_{1})^{-1}(Q(y_{2})-Q(y_{1})))$. \end{lem}
\begin{proof}
For the first assertion, a sequence of $\ell$ successive integrations
by parts shows that 
\begin{eqnarray*}
\Gamma_{\ell}(y_{1},y_{2}) & = & \left[E_{1}(y)(Q(y)-Q(y_{1}))^{\ell}\right]_{y_{1}}^{y_{2}}-\ell\int_{y_{1}}^{y_{2}}E_{1}(y)q(y)(Q(y)-Q(y_{1}))^{\ell-1}dy\\
 & = & S(y_{1},y_{2})+(-1)^{\ell}\ell!\int_{y_{1}}^{y_{2}}E_{\ell}(y)q(y)dy\\
 & = & S(y_{1},y_{2})+(-1)^{\ell}\ell!(E_{\ell+1}(y_{2})-E_{\ell+1}(y_{1}))
\end{eqnarray*}
where $E_{k}$ is defined recursively by $E_{1}(y)=\int_{0}^{y}p(\tau)d\tau$
and $E_{k+1}(y)=\int_{0}^{y}E_{k}(\tau)q(\tau)d\tau$, and where $S(y_{1},y_{2})\in(Q(y_{2})-Q(y_{1}))A[y_{1},y_{2}]$.
So $\int_{y_{1}}^{y_{2}}p(y)(Q(y)-Q(y_{1}))^{r}dy$ belongs to $(x^{n},Q(y_{2})-Q(y_{1}))A[y_{1},y_{2}]$
if and only if $E_{\ell+1}(y_{2})-E_{\ell+1}(y_{1})$ belongs to this
ideal. 

Since the residue class of $Q(y)\in A[y]$ in $\kappa[y]$ is not
constant, it follows from the local criterion for flatness that $A[y]$
is a faithfully flat algebra over $A[Q(y)]$. By faithfully flat descent,
this implies in turn that the sequence 
\[
A[Q(y)]\hookrightarrow A[y]\stackrel{\cdot\otimes1-1\otimes\cdot}{\longrightarrow}A[y]\otimes_{A[\tau]}A[y]
\]
is exact whence, using the natural identification $A[y]\otimes_{A[\tau]}A[y]\simeq A[y_{1},y_{2}]/(Q(y_{2})-Q(y_{1}))$,
that a polynomial $F\in A[y]$ with $F(y_{2})-F(y_{1})$ belonging
to the ideal $(Q(y_{2})-Q(y_{1}))A[y_{1},y_{2}]$ has the form $F(y)=G(Q(y))$
for a certain polynomial $G(\tau)\in A[\tau]$. Thus $E_{\ell+1}(y_{2})-E_{\ell+1}(y_{1})$
belongs to $(x^{n},Q(y_{2})-Q(y_{1}))A[y_{1},y_{2}]$, if and only
if $E_{\ell+1}(y)$ is of the form $G(Q(y))+x^{n}R_{\ell+1}(y)$ for
some $G(\tau),R_{\ell+1}(\tau)\in A[\tau]$. This implies in turn
that $E_{\ell}(y)q(y)=G'(Q(y))q(y)+x^{n}R_{\ell+1}'(y)$ whence, since
$q(y)\in A[y]\setminus\mathfrak{m}A[y]$ is not a zero divisor modulo
$x^{n}$, that $E_{\ell}(y)=G'(Q(y))+x^{n}R_{\ell}(y)$ for a certain
$R_{\ell}(\tau)\in A[\tau]$. We conclude by induction that $E_{1}(y)=G^{(\ell+1)}(Q(y))+x^{n}R_{1}(y)$
and finally that $p(y)=G^{(\ell+2)}(Q(y))q(y)+x^{n}R(y)$ for a certain
$R(\tau)\in A[\tau]$. This proves a). 

The second assertion is clear in the case where $p(y)\in\mathfrak{m}A[y]$.
Otherwise, if $p(y)\in A[y]\setminus\mathfrak{m}A[y]$ then reducing
modulo $x$ and passing to the algebraic closure $\overline{\kappa}$
of $\kappa$, it is enough to show that if $q(y)\in\overline{\kappa}[y]$
is not constant and $p(y)\in\overline{\kappa}[y]$ is a nonzero polynomial
then for every $\ell\geq1$, the affine curves $C$ and $D$ in $\mathbb{A}_{\overline{\kappa}}^{2}=\mathrm{Spec}(\overline{\kappa}[y_{1},y_{2}])$
defined by the vanishing of the polynomials $\Theta(y_{1},y_{2})=(y_{2}-y_{1})^{-\ell-1}\int_{y_{1}}^{y_{2}}p(y)(Q(y)-Q(y_{1}))^{\ell}dy$
and $R(y_{1},y_{2})=(y_{2}-y_{1})^{-1}\int_{y_{1}}^{y_{2}}q(y)dy$
respectively always intersect each other. Suppose on the contrary
that $C\cap D=\emptyset$ and let $m=\deg q\geq1$ and $d=\deg p\geq0$.
Then the closures $\overline{C}$ and $\overline{D}$ of $C$ and
$D$ respectively in $\mathbb{P}_{\overline{\kappa}}^{2}=\mathrm{Proj}(\overline{\kappa}[y_{1},y_{2},y_{3}])$
intersect each others along a closed sub-scheme $Y$ of length $\deg\overline{C}\cdot\deg\overline{D}=m(d+\ell m)$
supported on the line $\{y_{3}=0\}\simeq\mathrm{Proj}(\overline{\kappa}[y_{1},y_{2}])$.
By definition, up to multiplication by a nonzero scalar, the top homogeneous
components of $R$ and $\Theta$ have the form $\prod_{i=1}^{m}(y_{2}-\zeta^{i}y_{1})$,
where $\zeta\in\overline{\kappa}$ is a primitive $(m+1)$-th root
of unity, and $(y_{2}-y_{1})^{\ell-1}\int_{y_{1}}^{y_{2}}y^{d}(y^{m+1}-y_{1}^{m+1})^{\ell}dy$
respectively. But on the other hand, we have for every $i=1,\ldots,m$ 

\[
\overline{\kappa}[y_{2}]/(y_{2}-\zeta^{i},(y_{2}-1)^{-r-1}\int_{1}^{y_{2}}y^{d}(y^{m+1}-1)^{r}dy)\simeq\overline{\kappa}[y_{2}]/(y_{2}-\zeta^{i},(\zeta^{i}-1)^{-r-1}\int_{1}^{\zeta^{i}}\tau^{d}(\tau^{m+1}-1)^{r}d\tau),
\]
and hence the length of the above algebra is either $1$ or $0$ depending
on whether $\int_{1}^{\zeta^{i}}\tau^{d}(\tau^{m+1}-1)d\tau\in\overline{\kappa}$
is zero or not. This implies that the length of $Y$ is at most equal
to $m$ and so the only possibility would be that $d=0$ and $\ell=m=1$,
i.e. $C$ and $D$ are parallel lines in $\mathbb{A}_{\overline{\kappa}}^{2}$.
But since $\int_{1}^{-1}(\tau^{2}-1)d\tau\neq0$, this last possibility
is also excluded. 
\end{proof}

\section{\label{sec:Twin-Triviality} Global equivariant triviality of twin-triangular
actions}

By virtue of Proposition \ref{prop:Main-Local}, every proper triangular
$\mathbb{G}_{a,S}$-action on $\mbox{\ensuremath{\sigma}:}\mathbb{G}_{a,S}\times_{S}\mathbb{A}_{S}^{3}\rightarrow\mathbb{A}_{S}^{3}$
on $\mathbb{A}_{S}^{3}$ is conjugate to one generated by a twin-triangular
$A$-derivation $\partial$ of $A[y,z_{+},z_{-}]$ of the form 
\[
\partial=x^{n}\partial_{y}+p_{+}(y)\partial_{z_{+}}+p_{-}(y)\partial_{z_{-}}
\]
for certain polynomials $p_{\pm}(y)\in A[y]$. So to complete the
proof of the Main Theorem, it remains to show the following generalization
of the main result in \cite{DubFin11}:
\begin{prop}
\label{prop:Twin-Loc-trivi} Let $S$ be the spectrum of discrete
valuation $A$ containing a field of characteristic $0$. Then a proper
twin-triangular $\mathbb{G}_{a,S}$-action on $\mathbb{A}_{S}^{3}$
has affine geometric quotient $\mathfrak{X}=\mathbb{A}_{S}^{3}/\mathbb{G}_{a,S}$. 
\end{prop}
\begin{parn} The principle of the proof given below is the following:
we exploit the twin triangularity to construct two $\mathbb{G}_{a,S}$-invariant
principal open subsets $W_{\Gamma_{+}}$ and $W_{\Gamma_{-}}$ in
$\mathbb{A}_{S}^{3}$ with the property that the union of corresponding
principal open subspaces $\mathfrak{X}_{\Gamma_{\pm}}=W_{\Gamma_{\pm}}/\mathbb{G}_{a,S}$
of $\mathfrak{X}$ covers the closed fiber of the structure morphism
$\mathrm{p}:\mathfrak{X}\rightarrow S$. We then show that $\mathfrak{X}_{\Gamma_{+}}$
and $\mathfrak{X}_{\Gamma_{-}}$ are in fact affine sub-schemes of
$\mathfrak{X}$. On the other hand, since $\partial$ admits $x^{-n}y$
as a global slice over $A_{x}$, the generic fiber of $\mathrm{p}$
is isomorphic to the affine plane over the function field $A_{x}$
of $S$. So it follows that $\mathfrak{X}$ is covered by three principal
affine open sub-schemes $\mathfrak{X}_{\Gamma_{+}}$, $\mathfrak{X}_{\Gamma_{-}}$
and $\mathfrak{X}_{x}$ corresponding to regular functions $x$, $\Gamma_{+}$,
$\Gamma_{-}$ which generate the unit ideal in $\Gamma(\mathfrak{X},\mathcal{O}_{\mathfrak{X}})\simeq A[y,z_{+},z_{-}]^{\mathbb{G}_{a,S}}\subset A[y,z_{+},z_{-}]$,
whence is an affine scheme.

\end{parn}

\begin{parn} The fact that the affineness of $\mathrm{p}:\mathfrak{X}=\mathbb{A}_{S}^{3}/\mathbb{G}_{a,S}\rightarrow S=\mathrm{Spec}(A)$
is a local property with respect to the fpqc topology on $S$ \cite[VIII.5.6]{SGA1}
enables a reduction to the case where the discrete valuation ring
$A$ is Henselian or complete. Since it contains a field of characteristic
zero, an elementary application of Hensel's Lemma implies that a maximal
subfield of such a local ring $A$ is a field of representatives,
i.e. a subfield which is mapped isomorphically by the quotient projection
$A\mapsto A/\mathfrak{m}$ onto the residue field $\kappa=A/\mathfrak{m}$.
This is in fact the only property of $A$ that we will use in the
sequel. So from now on, $(A,\mathfrak{m},\kappa)$ is a discrete valuation
ring containing a field $\kappa$ of characteristic $0$ and with
residue field $A/\mathfrak{m}\simeq\kappa$.

\end{parn}

\subsection{Twin-triangular actions in general position and associated invariant
covering}

Here we construct a pair of principal $\mathbb{G}_{a,S}$-invariant
open subsets $W_{\pm}=W_{\Gamma_{\pm}}$ of $\mathbb{A}_{S}^{3}$
associated with a twin-triangular $A$-derivation of $A[y,z_{+},z_{-}]$
whose geometric quotients will be studied in the next sub-section.
We begin with a technical condition which will be used to guarantee
that the union of $W_{+}$ and $W_{-}$ covers the closed fiber of
the projection $\mathrm{pr}_{S}:\mathbb{A}_{S}^{3}\rightarrow S$. 
\begin{defn}
\label{def:general_position} Let $(A,\mathfrak{m},\kappa)$ be a
discrete valuation valuation ring containing a field of characteristic
$0$ and let $x\in\mathfrak{m}$ be a uniformizing parameter. A twin-triangular
$A$-derivation $\partial=x^{n}\partial_{y}+p_{+}(y)\partial_{z_{+}}+p_{-}(y)\partial_{z_{-}}$
of $A[y,z_{+},z_{-}]$ is said to be in \emph{general position} if
it satisfies the following properties: 

a) The residue classes $\overline{p}_{\pm}\in\kappa[y]$ of the polynomials
$p_{\pm}\in A[y]$ modulo $\mathfrak{m}$ are both non zero and relatively
prime,

b) There exist integrals $\overline{P}_{\pm}\in A[y]$ of $\overline{p}_{\pm}$
with respect to $y$ for which the inverse images of the branch loci
of the morphisms $\overline{P}_{+}:\mathbb{A}_{\kappa}^{1}\rightarrow\mathbb{A}_{\kappa}^{1}$
and $\overline{P}_{-}:\mathbb{A}_{\kappa}^{1}\rightarrow\mathbb{A}_{\kappa}^{1}$
are disjoint. \end{defn}
\begin{lem}
\label{lem:Bad-Plane-Removal} With the notation above, every twin-triangular
$A$-derivation $\partial$ of $A[y,z_{+},z_{-}]$ generating a fixed
point free $\mathbb{G}_{a,S}$-action on $\mathbb{A}_{S}^{3}$ is
conjugate to one in general position. \end{lem}
\begin{proof}
A twin-triangular derivation $\partial=x^{n}\partial_{y}+p_{+}(y)\partial_{z_{+}}+p_{-}(y)\partial_{z_{-}}$
generates a fixed point free $\mathbb{G}_{a,S}$-action if and only
if $x^{n}$, $p_{+}(y)$ and $p_{-}(y)$ generate the unit ideal in
$A[y,z_{+},z_{-}]$. So the residue classes $\overline{p}_{+}$ and
$\overline{p}_{-}$ of $p_{+}$ and $p_{-}$ are relatively prime
and at least one of them, say $\overline{p}_{-}$, is nonzero. If
$\overline{p}_{+}=0$ then $p_{-}$ is necessarily of the form $p_{-}(y)=c+x\tilde{p}_{-}(y)$
for some $c\in A^{*}$ and so changing $z_{+}$ for $z_{+}+z_{-}$
yields a twin-triangular derivation conjugate to $\partial$ for which
the corresponding polynomials $p_{\pm}(y)$ both have non zero residue
classes modulo $x$. More generally, changing $z_{-}$ for $az_{-}+bz_{+}$
for general $a\in A^{*}$ and $b\in A$ yields a twin-triangular derivation
conjugate to $\partial$ and still satisfying condition a) in Definition
\ref{def:general_position}. So it remains to show that up to such
a coordinate change, condition b) in the Definition can be achieved. 

This can be seen as follows : we consider $\mathbb{A}_{\kappa}^{2}$
embedded in $\mathbb{P}_{\kappa}^{2}={\rm Proj}(\kappa[u,v,w])$ as
the complement of the line $L_{\infty}=\left\{ w=0\right\} $ so that
the coordinate system $\left(u,v\right)$ on $\mathbb{A}^{2}$ is
induced by the projections from the $\kappa$-rational points $\left[0:1:0\right]$
and $\left[1:0:0\right]$ respectively. We let $C$ be the closure
in $\mathbb{P}^{2}$ of the image of the morphism $j=(\overline{P}_{+},\overline{P}_{-}):\mathbb{A}_{\kappa}^{1}={\rm Spec}(\kappa[y])\rightarrow\mathbb{A}_{\kappa}^{2}$
defined by the residue classes $\overline{P}_{+}$ and $\overline{P}_{-}$
in $\kappa[y]$ of integrals $P_{\pm}(y)\in A[y]$ of $p_{\pm}(y)$,
and we denote by $Z\subset C$ the image by $j$ of the inverse image
of the branch locus of $\overline{P}_{+}:\mathbb{A}_{\kappa}^{1}\rightarrow\mathbb{A}_{\kappa}^{1}$.
Note that $Z$ is a finite subset of $C$ defined over $\kappa$.
Since the condition that a line through a fixed point in $\mathbb{P}_{\kappa}^{2}$
intersects transversally a fixed curve is Zariski open, the set of
lines in $\mathbb{P}_{\kappa}^{2}$ passing through a point of $Z$
and tangent to a local analytic branch of $C$ at some point is finite.
Therefore, the complement of the finitely many intersection points
of these lines with $L_{\infty}$ is a Zariski open subset $U$ of
$L_{\infty}$ with the property that for every $q\in U$, the line
through $q$ and every arbitrary point of $Z$ intersects every local
analytic branch of $C$ transversally at every point. By construction,
the rational projections from $\left[0:1:0\right]$ and an arbitrary
$\kappa$-rational point in $U\setminus\{\left[0:1:0\right]\}$ induce
a new coordinate system on $\mathbb{A}_{\kappa}^{2}$ of the form
$\left(u,av+bu\right)$, $a\neq0$, with the property that $Z$ is
not contained in the inverse image of the branch locus of the morphism
$a\overline{P}_{-}+b\overline{P}_{+}:\mathbb{A}_{\kappa}^{1}\rightarrow\mathbb{A}_{\kappa}^{1}$.
Changing $z_{-}$ for $az_{-}+bz_{+}$ for a pair $(a,b)\in\kappa^{*}\times\kappa\subset A^{*}\times A$
corresponding to a general point in $U$ yields a twin-triangular
derivation conjugate to $\partial$ and satisfying simultaneously
conditions a) and b) in Definition \ref{def:general_position}. 
\end{proof}
\begin{parn} \label{par:finite_etale_restriction} Now let $\partial=x^{n}\partial_{y}+p_{+}(y)\partial_{z_{+}}+p_{-}(y)\partial_{z_{-}}$
be a twin-triangular $A$-derivation of $A[y,z_{+},z_{-}]$ generating
a proper whence fixed point free $\mathbb{G}_{a,S}$-action $\sigma:\mathbb{G}_{a,S}\times_{S}\mathbb{A}_{S}^{3}\rightarrow\mathbb{A}_{S}^{3}$.
By virtue of Lemma \ref{lem:Bad-Plane-Removal} above, we may assume
up to a coordinate change preserving twin-triangularity that $\partial$
is in general position. Property a) in Definition \ref{def:general_position}
then guarantees in particular that the triangular derivations $\partial_{\pm}=x^{n}\partial_{y}+p_{\pm}(y)\partial_{z_{\pm}}$
of $A[y,z_{\pm}]$ are both irreducible. Furthermore, given any integral
$P_{\pm}(y)\in A[y]$ of $p_{\pm}(y)$, the morphism $\overline{P}_{\pm}:\mathbb{A}_{\kappa}^{1}\rightarrow\mathbb{A}_{\kappa}^{1}$
obtained by restricting $P_{\pm}:\mathbb{A}_{S}^{1}={\rm Spec}(A[y])\rightarrow\mathbb{A}_{S}^{1}={\rm Spec}(A[t])$
to the closed fiber of $\mathrm{pr}_{S}:\mathbb{A}_{S}^{3}\rightarrow S$
is not constant. The branch locus of $\overline{P}_{\pm}$ is then
a principal divisor $\mathrm{div}(\alpha_{\pm}(t))$ for a certain
polynomial $\alpha_{\pm}(t)\in\kappa[t]\subset A[t]$ generating the
kernel of the homomorphism $\kappa[t]\rightarrow\kappa[y]/(\overline{p}_{\pm}(y))$,
$t\mapsto\overline{P}_{\pm}(y)+(\overline{p}_{\pm}(y))$. Property
b) in Definition \ref{def:general_position} guarantees that we can
choose $P_{+}$ and $P_{-}$ in such a way that the polynomial $\alpha_{+}(\overline{P}_{+}(y))$
and $\alpha_{-}(\overline{P}_{-}(y))$ generate the unit ideal in
$\kappa[y]$. Up to a diagonal change of coordinates on $A[y,z_{+},z_{-}]$,
we may further assume without loss of generality that $\overline{P}_{+}$
and $\overline{P}_{-}$ are monic. 

\end{parn}

\begin{parn} \label{par:open_cover_def} We let $R_{\pm}=A[t]_{\alpha_{\pm}}$
and we let $U_{\pm}=\mathrm{Spec}(R_{\pm})$ be the principal open
subset of $\mathbb{A}_{S}^{1}=\mathrm{Spec}(A[t])$ where $\alpha_{\pm}$
does not vanish. The polynomial $\Phi_{\pm}=-x^{n}z_{\pm}+P_{\pm}(y)\in A[y,z_{+},z_{-}]$
belongs to the kernel of $\partial$ hence defines a $\mathbb{G}_{a,S}$-invariant
morphism $\Phi_{\pm}:\mathbb{A}_{S}^{3}=\mathrm{Spec}(A[y,z_{+},z_{-}])\rightarrow\mathbb{A}_{S}^{1}=\mathrm{Spec}(A[t])$.
We let 
\begin{eqnarray*}
W_{\pm} & = & \Phi_{\pm}^{-1}(U_{\pm})\simeq\mathrm{Spec}(R_{\pm}[y,z_{+},z_{-}]/(-x^{n}z_{\pm}+P_{\pm}(y)-t))
\end{eqnarray*}
Note that $W_{\pm}$ is a $\mathbb{G}_{a,S}$-invariant open subset
of $\mathbb{A}_{S}^{3}$ which can be identified with the principal
open subset where the $\mathbb{G}_{a,S}$-invariant regular function
$\Gamma_{\pm}=\alpha_{\pm}\circ\Phi_{\pm}$ does not vanish. Since
$\alpha_{+}(\overline{P}_{+}(y))$ and $\alpha_{-}(\overline{P}_{-}(y))$
generate the unit ideal in $\kappa[y]$, it follows that the union
of $W_{+}$ and $W_{-}$ covers the closed fiber of the projection
$\mathrm{pr}_{S}:\mathbb{A}_{S}^{3}\rightarrow S$. 

\end{parn}

\subsection{Affineness of geometric quotients}

With the notation of \S \ref{par:open_cover_def} above, the geometric
quotient $\mathfrak{X}_{\pm}=W_{\pm}/\mathbb{G}_{a,S}$ for the action
induced by $\sigma:\mathbb{G}_{a,S}\times_{S}\mathbb{A}_{S}^{3}\rightarrow\mathbb{A}_{S}^{3}$
can be identified with the principal open sub-space $\mathfrak{X}_{\Gamma_{\pm}}$
of $\mathfrak{X}=\mathbb{A}_{S}^{3}/\mathbb{G}_{a,S}$ where the invariant
function $\Gamma_{\pm}\in A[y,z_{+},z_{-}]^{\mathbb{G}_{a,S}}\simeq\Gamma(\mathfrak{X},\mathcal{O}_{\mathfrak{X}})$
does not vanish. The properness of $\sigma$ implies that $\mathfrak{X}$,
whence $\mathfrak{X}_{+}$ and $\mathfrak{X}_{-}$, are separated
algebraic spaces, and the construction of $W_{+}$ and $W_{-}$ guarantees
that the closed fiber of the structure morphism $\mathrm{p}:\mathfrak{X}\rightarrow S$
is contained in the union of $\mathfrak{X}_{+}$ and $\mathfrak{X}_{-}$.
So to complete the proof of Proposition \ref{prop:Twin-Loc-trivi},
it remains to show that $\mathfrak{X}_{\pm}$ is an affine scheme.
In fact, since $\mathfrak{X}_{\pm}$ is by construction an algebraic
space over the affine scheme $U_{\pm}=\mathrm{Spec}(R_{\pm})$, its
affineness is equivalent to that of the structure morphism $q_{\pm}:\mathfrak{X}_{\pm}\rightarrow U_{\pm}$,
a property which can be checked locally with respect to the \'etale
topology on $U_{\pm}$. 

\begin{parn} In our situation, there is a natural finite \'etale
base change $\varphi_{\pm}:\tilde{U}_{\pm}\rightarrow U_{\pm}$ which
is obtained as follows: By construction, see \S \ref{par:finite_etale_restriction}
above, the morphism $\overline{P}_{\pm}:\mathbb{A}_{\kappa}^{1}=\mathrm{Spec}(\kappa[y])\rightarrow\mathrm{Spec}(\kappa[t])$,
restricts to a finite \'etale covering $h_{0,\pm}:C_{1,\pm}={\rm Spec}(\kappa[y]_{\alpha_{\pm}(\overline{P}_{\pm}(y))})\rightarrow C_{\pm}={\rm Spec}(\kappa[t]_{\alpha_{\pm}(t)})$
of degree $r_{\pm}=\deg_{y}(\overline{P}_{\pm}(y))$. Letting $\tilde{C}_{\pm}=\mathrm{Spec}(B_{\pm})$
be the normalization of $C_{\pm}$ in the Galois closure $L_{\pm}$
of the field extension $i_{\pm}:\kappa(t)\hookrightarrow\kappa(y)$,
the induced morphism $h_{\pm}:\tilde{C}_{\pm}\rightarrow C_{\pm}$
is an \'etale Galois cover with Galois group $G_{\pm}=\mathrm{Gal}(L_{\pm}/\kappa(t))$,
which factors as 
\[
h_{\pm}:\tilde{C}_{\pm}=\mathrm{Spec}(B_{\pm})\stackrel{h_{1,\pm}}{\longrightarrow}C_{1,\pm}={\rm Spec}(\kappa[y]_{\alpha_{\pm}(\overline{P}_{\pm}(y))})\stackrel{h_{0,\pm}}{\longrightarrow}C_{\pm}={\rm Spec}(\kappa[t]_{\alpha_{\pm}(t)})
\]
where $h_{1,\pm}:\tilde{C}_{\pm}\rightarrow C_{1,\pm}$ is an \'etale
Galois cover for a certain subgroup $H_{\pm}$ of $G_{\pm}$ of index
$r_{\pm}$. Letting $\tilde{R}_{\pm}=A\otimes_{\kappa}B_{\pm}\simeq A[t]_{\alpha_{\pm}(t)}\otimes_{\kappa[t]_{\alpha_{\pm}(t)}}B_{\pm}$
and $\tilde{U}_{\pm}=\mathrm{Spec}(\tilde{R}_{\pm})$, the morphism
$\varphi_{\pm}=\mathrm{pr}_{1}:\tilde{U}_{\pm}\simeq U_{\pm}\times_{C_{\pm}}\tilde{C}_{\pm}\rightarrow U_{\pm}$
is an \'etale Galois cover with Galois group $G_{\pm}$, in particular
a finite morphism. Since $\mathfrak{X}_{\pm}$ is separated, the algebraic
space $\tilde{\mathfrak{X}}_{\pm}=\mathfrak{X}_{\pm}\times_{U_{\pm}}\tilde{U}_{\pm}$
is separated and, by construction, isomorphic to the geometric quotient
of the scheme
\begin{eqnarray*}
\tilde{W}_{\pm}=W_{\pm}\times_{U_{\pm}}\tilde{U}_{\pm} & \simeq & \mathrm{Spec}(\tilde{R}_{\pm}[y,z_{+},z_{-}]/(-x^{n}z_{\pm}+P_{\pm}(y)-t))
\end{eqnarray*}
by the proper $\mathbb{G}_{a,\tilde{U}_{\pm}}$-action generated by
the locally nilpotent $\tilde{R}_{\pm}$-derivation $x^{n}\partial_{y}+p_{+}(y)\partial_{z_{+}}+p_{-}(y)\partial_{z_{-}}$
of $\tilde{R}_{\pm}[y,z_{+},z_{-}]//(-x^{n}z_{\pm}+P_{\pm}(y)-t)$,
which commutes with the action of $G_{\pm}$. The following Lemma
completes the proof of Proposition \ref{prop:Twin-Loc-trivi} whence
of the Main Theorem. 

\end{parn}
\begin{lem}
The geometric quotient $\tilde{\mathfrak{X}}_{\pm}=\tilde{W}_{\pm}/\mathbb{G}_{a,\tilde{U}_{\pm}}$
is an affine $\tilde{U}_{\pm}$-scheme. \end{lem}
\begin{proof}
Since $\tilde{U}_{\pm}$ is affine, the assertion is equivalent to
the affineness of $\tilde{\mathfrak{X}}_{\pm}$. From now on, we only
consider the case of $\tilde{\mathfrak{X}}_{+}=\tilde{W}_{+}/\mathbb{G}_{a,\tilde{U}_{+}}$,
the case of $\tilde{\mathfrak{X}}_{-}$ being similar. To simplify
the notation, we drop the corresponding subscript ``$+$'', writing
simply $\tilde{W}=\mathrm{Spec}(\tilde{R}[y,z,z_{-}]/(-x^{n}z+P(y)-t))$.
We denote $x\otimes1\in\tilde{R}=A\otimes_{\kappa}B$ by $x$ and
we further identify $B$ with a sub-$\kappa$-algebra of $\tilde{R}$
via the homomorphism $1\otimes\mathrm{id}_{B}:B\rightarrow\tilde{R}$
and with the quotient $\tilde{R}/x\tilde{R}$ via the composition
$1\otimes\mathrm{id}_{B}:B\rightarrow A\otimes_{\kappa}B\rightarrow A\otimes_{\kappa}B/((x\otimes1)A\otimes_{\kappa}B)=\kappa\otimes_{\kappa}B\simeq B$. 

By construction of $B$, the monic polynomial $\overline{P}(y)-t\in B\left[y\right]$
splits as $\overline{P}(y)-t=\prod_{\overline{g}\in G/H}(y-t_{\overline{g}})$
for certain elements $t_{\overline{g}}\in B$, $\overline{g}\in G/H$,
on which the Galois group $G$ acts by permutation $g'\cdot t_{\overline{g}}=t_{\overline{(g')^{-1}\cdot g}}$.
Furthermore, since $h_{0}:C_{1}\rightarrow C$ is \'etale, it follows
that for distinct $\overline{g},\overline{g}'\in G/H$, $t_{\overline{g}}-t_{\overline{g'}}\in B$
is an invertible regular function on $\tilde{C}$ whence on $\tilde{U}=S\times_{\mathrm{Spec}(\kappa)}\tilde{C}$
via the identifications made above. This implies in turn that there
exists a collection of elements $\sigma_{\overline{g}}\in\tilde{R}$
with respective residue classes $t_{\overline{g}}\in B=\tilde{R}/x\tilde{R}$
modulo $x$, $\overline{g}\in G/H$, on which $G$ acts by permutation,
a $G$-invariant polynomial $S_{1}\in\tilde{R}\left[y\right]$ with
invertible residue class modulo $x$ and a $G$-invariant polynomial
$S_{2}\in\tilde{R}\left[y\right]$ such that in $\tilde{R}\left[y\right]$
one can write 
\[
P(y)-t=S_{1}(y)\prod_{\overline{g}\in G/H}(y-\sigma_{\overline{g}})+x^{n}S_{2}(y).
\]
Concretely, the elements $\sigma_{\overline{g}}=\sigma_{\overline{g},n-1}\in\tilde{R}$,
$\overline{g}\in G/H$, can be constructed by induction via a sequence
of elements $\sigma_{\overline{g},m}\in\tilde{R}$, $\overline{g}\in G/H$,
$m=0,\ldots,n-1$, starting with $\sigma_{\overline{g},0}=t_{\overline{g}}\in B\subset\tilde{R}$
and culminating in $\sigma_{\overline{g},n-1}=\sigma_{\overline{g}}$,
and characterized by the property that for every $m=0,\ldots,n-1$,
there exists $\mu_{\overline{g},m}\in\tilde{R}$ such that $P(\sigma_{\overline{g},m})-t=x^{m+1}\mu_{\overline{g},m}$,
$\overline{g}\in G/H$. Indeed, writing $P(y)-t=\prod_{\overline{g}\in G/H}(y-t_{\overline{g}})+x\tilde{P}(y)$
for a certain $\tilde{P}(y)\in\tilde{R}[y]$ and assuming that the
$\sigma_{\overline{g},m}$, $\overline{g}\in G/H$, have been constructed
up to a certain index $m<n-1$, we look for elements $\sigma_{\overline{g},m+1}\in\tilde{R}$
written in the form $\sigma_{\overline{g},m}+x^{m+1}\lambda_{\overline{g}}$
for some $\lambda_{\overline{g}}\in\tilde{R}$. For a fixed $\overline{g}_{0}\in G/H$,
the conditions impose that 

\begin{eqnarray*}
P(\sigma_{\overline{g}_{0},m+1})-t & = & \prod_{\overline{g}\in G/H}(\sigma_{\overline{g}_{0},m}+x^{m+1}\lambda_{\overline{g}_{0}}-t_{\overline{g}})+x\tilde{P}(\sigma_{\overline{g}_{0},m}+x^{m+1}\lambda_{\overline{g}_{0}})\\
 & = & x^{m+1}\lambda_{\overline{g}_{0}}\prod_{\overline{g}\in(G/H)\setminus\{\overline{g}_{0}\}}(t_{\overline{g}_{0}}-t_{\overline{g}})+P(\sigma_{\overline{g}_{0},m})-t+x^{m+2}\nu_{\overline{g}_{0},m}\\
 & = & x^{m+1}\lambda_{\overline{g}_{0}}\prod_{\overline{g}\in(G/H)\setminus\{\overline{g}_{0}\}}(t_{\overline{g}_{0}}-t_{\overline{g}})+x^{m+1}\mu_{\overline{g}_{0,}m}+x^{m+2}\nu_{\overline{g}_{0},m}
\end{eqnarray*}
for some $\nu_{\overline{g}_{0},m}\in\tilde{R}$, and since $\prod_{\overline{g}\in(G/H)\setminus\{\overline{g}_{0}\}}(t_{\overline{g}_{0}}-t_{\overline{g}})\in\tilde{R}^{*}$,
we conclude that 
\[
\lambda_{\overline{g}_{0}}=\frac{\mu_{\overline{g}_{0},m}}{\prod_{\overline{g}\in(G/H)\setminus\{\overline{g}_{0}\}}(t_{\overline{g}_{0}}-t_{\overline{g}})}\quad\textrm{and}\quad\mu_{\overline{g}_{0},m+1}=\nu_{\overline{g}_{0},m}.
\]
A direct computation shows further that $g'\cdot\sigma_{\overline{g},m+1}=\sigma_{\overline{(g')^{-1}\cdot g},m+1}$
and that $g'\cdot\mu_{\overline{g},m+1}=\mu_{\overline{(g')^{-1}\cdot g},m+1}$.
Iterating this procedure $n-1$ times yields the desired collection
of elements $\sigma_{\overline{g}}=\sigma_{\overline{g},n-1}\in\tilde{R}$.
By construction, $\prod_{\overline{g}\in G/H}(y-\sigma_{\overline{g}})\in\tilde{R}[y]$
is then an invariant polynomial which divides $P(y)-t$ modulo $x^{n}\tilde{R}$,
which implies in turn the existence of the $G$-invariant polynomials
$S_{1}(y),S_{2}(y)\in\tilde{R}[y]$. 

The closed fiber of the induced morphism $\tilde{W}\rightarrow S$
consists of a disjoint union of closed sub-schemes $D_{\overline{g}}\simeq\mathrm{Spec}(\tilde{R}[z,z_{-}])\simeq\mathbb{A}_{\tilde{C}}^{2}$
with defining ideals $(x,y-\sigma_{\overline{g}})$, $\overline{g}\in G/H$.
The open sub-scheme $\tilde{W}_{\overline{g}}=\tilde{W}\setminus\bigcup_{\overline{g}'\in(G/H)\setminus\{\overline{g}\}}D_{\overline{g}'}$
of $\tilde{W}$ is $\mathbb{G}_{a,\tilde{U}}$-invariant and one checks
using the above expression for $P(y)-t$ that the rational map 
\[
\tilde{W}\dashrightarrow\mathrm{Spec}(\tilde{R}[u_{\overline{g}},z_{-}]),\quad(y,z,z_{-})\mapsto(u_{\overline{g}},z_{-})=(\frac{y-\sigma_{\overline{g}}}{x^{n}}=\frac{z-S_{2}(y)}{S_{1}(y)\prod_{\overline{g}'\in(G/H)\setminus\{\overline{g}\}}(y-\sigma_{\overline{g}'})},z_{-})
\]
induces a $\mathbb{G}_{a,\tilde{U}}$-equivariant isomorphism $\tau_{g}:\tilde{W}_{\overline{g}}\stackrel{\sim}{\rightarrow}\mathbb{A}_{\tilde{U}}^{2}=\mathrm{Spec}(\tilde{R}[u_{\overline{g}},z_{-}])$
for the $\mathbb{G}_{a,\tilde{U}}$-action on $\mathbb{A}_{\tilde{U}}^{2}$
generated by the locally nilpotent $\tilde{R}$-derivation $\partial_{u_{\overline{g}}}+p_{-}(x^{n}u_{\overline{g}}+\sigma_{\overline{g}})\partial_{z_{-}}$
of $\tilde{R}[u_{\overline{g}},z_{-}]$. The latter is a translation
with $u_{\overline{g}}$ as a global slice and with geometric quotient
$\tilde{W}_{\overline{g}}/\mathbb{G}_{a,\tilde{U}}$ isomorphic to
$\mathrm{Spec}(\tilde{R}[v_{\overline{g}}])$ where 
\[
v_{\overline{g}}=z_{-}-x^{-n}(P_{-}(x^{n}u_{\overline{g}}+\sigma_{\overline{g}})-P_{-}(\sigma_{\overline{g}}))\in\tilde{R}[u_{\overline{g}},z_{-}]^{\mathbb{G}_{a,\tilde{U}}}.
\]
By construction, for distinct $\overline{g},\overline{g}'\in G/H$,
the rational functions $\tau_{\overline{g}}^{*}v_{\overline{g}}$
and $\tau_{\overline{g}'}^{*}v_{\overline{g}'}$ on $\tilde{W}$ differ
by the addition of the element 
\[
f_{\overline{g},\overline{g}'}=x^{-n}(P_{-}(\sigma_{\overline{g}})-P_{-}(\sigma_{\overline{g}'}))\in\tilde{R}_{x}\in\Gamma(\tilde{W}_{\overline{g}}\cap\tilde{W}_{\overline{g}'},\mathcal{O}_{\tilde{W}}).
\]
This implies that $\tilde{\mathfrak{X}}=\tilde{W}/\mathbb{G}_{a,\tilde{U}}$
is isomorphic to the $\tilde{U}$-scheme obtained by gluing $r$ copies
$\tilde{\mathfrak{X}}_{g}=\mathrm{Spec}(\tilde{R}[v_{\overline{g}}])$
of $\mathbb{A}_{\tilde{U}}^{1}$ along the principal open subsets
$\tilde{\mathfrak{X}}_{\overline{g},x}\simeq\mathrm{Spec}(\tilde{R}_{x}[v_{\overline{g}}])$
via the isomorphisms induced by the $\tilde{R}_{x}$-algebra isomorphisms
\[
\xi_{\overline{g},\overline{g}'}^{*}:\tilde{R}_{x}[v_{\overline{g}}]\rightarrow\tilde{R}_{x}[v_{\overline{g}'}],v_{\overline{g}}\mapsto v_{\overline{g}'}+f_{\overline{g},\overline{g}'},\quad\overline{g},\overline{g}'\in G/H,\;\overline{g}\neq\overline{g}'.
\]
Since by assumption $\tilde{\mathfrak{X}}$ is separated, it follows
from \cite[I.5.5.6]{EGA1} that for every pair of distinct elements
$\overline{g},\overline{g}'\in G/H$, the sub-ring $\tilde{R}[v_{\overline{g}'},f_{\overline{g},\overline{g}'}]$
of $\tilde{R}_{x}[v_{\overline{g}'}]$ generated by the union of $\tilde{R}[v_{\overline{g}'}]$
and $\xi_{\overline{g},\overline{g}'}^{*}(\tilde{R}[v_{\overline{g}}])$
is equal to $\tilde{R}_{x}[v_{\overline{g}'}]$. This holds if and
only if $\tilde{R}[f_{\overline{g},\overline{g}'}]=\tilde{R}_{x}$
whence if and only if $f_{\overline{g},\overline{g}'}\in\tilde{R}_{x}$
has the form $f_{\overline{g},\overline{g}'}=x^{-m_{\overline{g},\overline{g}'}}F_{\overline{g},\overline{g}'}$
for a certain $m_{\overline{g},\overline{g}'}>1$ and an element $F_{\overline{g},\overline{g}'}\in\tilde{R}$
with invertible residue class modulo $x$. 

This additional information enables a proof of the affineness of $\tilde{\mathfrak{X}}$
by induction on $r$ as follows: given a pair of distinct elements
$\overline{g},\overline{g}'\in G/H$ such that $m_{\overline{g},\overline{g}'}=m>0$
is maximal, we let $\theta_{\overline{g}}=0$ and $\theta_{\overline{g}''}=x^{m-m_{\overline{g},\overline{g}''}}F_{\overline{g},\overline{g}''}\in\tilde{R}$
for every $\overline{g}''\in(G/H)\setminus\{\overline{g}\}$. The
choice of the elements $\theta_{\overline{g}''}\in\tilde{R}$ guarantees
that the local sections 
\[
\psi_{\overline{g}''}=x^{m}v_{\overline{g}''}+\theta_{\overline{g}''}\in\Gamma(\tilde{\mathfrak{X}}_{\overline{g}''},\mathcal{O}_{\tilde{\mathfrak{X}}}),\quad\overline{g}''\in G/H
\]
glue to a global regular function $\psi\in\Gamma(\tilde{\mathfrak{X}},\mathcal{O}_{\tilde{\mathfrak{X}}})$.
Since $\theta_{\overline{g}'}=F_{\overline{g},\overline{g}'}$ is
invertible modulo $x$, the regular functions $x$, $\psi$ and $\psi-\theta_{\overline{g}'}$
generate the unit ideal in $\Gamma(\tilde{\mathfrak{X}},\mathcal{O}_{\tilde{\mathfrak{X}}})$.
The principal open subset $\tilde{\mathfrak{X}}_{x}$ of $\tilde{\mathfrak{X}}$
is isomorphic to $\tilde{\mathfrak{X}}_{\overline{g},x}\simeq\mathrm{Spec}(\tilde{R}_{x}[v_{\overline{g}}])$
for every $\overline{g}\in G/H$, hence is affine. On the other hand,
$\tilde{\mathfrak{X}}_{\psi}$ and $\tilde{\mathfrak{X}}_{\psi-\theta_{\overline{g}'}}$
are contained respectively in the open sub-schemes $\tilde{\mathfrak{X}}(\overline{g})$
and $\tilde{\mathfrak{X}}(\overline{g}')$ obtained by gluing only
the $r-1$ open subsets $\tilde{\mathfrak{X}}_{\overline{g}''}$ corresponding
to the elements $\overline{g}''$ in $\left(G/H\right)\setminus\{\overline{g}\}$
and $\left(G/H\right)\setminus\{\overline{g}'\}$ respectively. By
the induction hypothesis, the latter are both affine and hence $\tilde{\mathfrak{X}}_{\psi}$
and $\tilde{\mathfrak{X}}_{\psi-\theta_{\overline{g}'}}$ are affine
as well. This shows that $\tilde{\mathfrak{X}}$ is an affine scheme
and completes the proof. 
\end{proof}
\bibliographystyle{amsplain}

\end{document}